\documentclass[11pt]{article}

\usepackage[colorlinks=true, citecolor=teal, linkcolor=black]{hyperref}

\usepackage{vitercik}
\usepackage{enumitem}
\newcommand{\score}{\texttt{score}}

\newcommand{\lf}{\left\lfloor}
\newcommand{\rf}{\right\rfloor}
\newcommand{\LP}{\mathsf{LP}}
\newcommand{\IP}{\mathsf{IP}}
\newcommand{\IH}{\mathsf{IH}}

\newcommand{\I}{\mathsf{I}}

\newcommand{\cut}[1][]{\vec{\alpha}_{#1}^T\vec{x}\le\beta_{#1}}
\usepackage{subcaption}
\usepackage{tabularx}

\title{Structural Analysis of Branch-and-Cut and the Learnability of Gomory Mixed Integer Cuts}
\author{Maria-Florina Balcan\thanks{Computer Science Department, Machine Learning Department, Carnegie Mellon University. \texttt{ninamf@cs.cmu.edu}} \and Siddharth Prasad\thanks{Computer Science Department, Carnegie Mellon University. \texttt{sprasad2@cs.cmu.edu}} \and Tuomas Sandholm\thanks{Computer Science Department, Carnegie Mellon University, Optimized Markets, Inc., Strategic Machine, Inc., Strategy Robot, Inc. \texttt{sandholm@cs.cmu.edu}} \and Ellen Vitercik\thanks{Department of Electrical Engineering and Computer Sciences, UC Berkeley. \texttt{vitercik@berkeley.edu}}}
\date{}

\begin{document}

\maketitle

\begin{abstract}
The incorporation of cutting planes within the branch-and-bound algorithm, known as branch-and-cut, forms the backbone of modern integer programming solvers. These solvers are the foremost method for solving discrete optimization problems and thus have a vast array of applications in machine learning, operations research, and many other fields. Choosing cutting planes effectively is a major research topic in the theory and practice of integer programming. We conduct a novel structural analysis of branch-and-cut that pins down how every step of the algorithm is affected by changes in the parameters defining the cutting planes added to the input integer program. Our main application of this analysis is to derive sample complexity guarantees for using machine learning to determine which cutting planes to apply during branch-and-cut. These guarantees apply to infinite families of cutting planes, such as the family of Gomory mixed integer cuts, which are responsible for the main breakthrough speedups of integer programming solvers. We exploit geometric and combinatorial structure of branch-and-cut in our analysis, which provides a key missing piece for the recent generalization theory of branch-and-cut.
\end{abstract}

\section{Introduction}\label{sec:introduction}

Integer programming (IP) solvers are the most widely-used tools for solving discrete optimization problems. They have numerous applications in machine learning, operations research, and many other fields, including MAP inference~\citep{Kappes13:Towards}, combinatorial auctions~\citep{Sandholm13:Very-Large-Scale}, natural language processing~\citep{Khashabi16:Question}, neural network verification~\citep{Bunel18:Unified}, interpretable classification~\citep{Zeng17:Interpretable}, training of optimal decision trees~\citep{Bertsimas17:Optimal}, and optimal clustering~\citep{Miyauchi18:Exact}, among many others.

Under the hood, IP solvers use the tree-search algorithm branch-and-bound~\citep{Land60:Automatic} augmented with \emph{cutting planes}, known as \emph{branch-and-cut} (B\&C). A cutting plane is a linear constraint that is added to the LP relaxation at any node of the search tree. With a carefully selected cutting plane, the LP guidance can more efficiently lead B\&C to the globally optimal integral solution. Cutting planes, specifically the family of \emph{Gomory mixed integer cuts} which we study in this paper, are responsible for breakthrough speedups of modern IP solvers~\cite{Cornuejols07:Revival}. 

Successfully employing cutting planes can be challenging because there are infinitely many cuts to choose from and there are still many open questions about which cuts to employ when. A growing body of research has studied the use of machine learning for cut selection~\cite{Tang20:Reinforcement,Balcan21:Sample,Balcan21:Improved,Huang22:Learning}.
In this paper, we analyze a machine learning setting where there is an unknown distribution over IPs---for example, a distribution over a shipping company's routing problems. The learner receives a \emph{training set} of IPs sampled from this distribution which it uses to learn cut parameters with strong average performance over the training set (leading, for example, to small search trees). We provide \emph{sample complexity bounds} for this procedure, which bound the number of training instances sufficient to ensure that if a set of cut parameters leads to strong average performance over the training set, it will also lead to strong expected performance on future IPs from the same distribution. These guarantees apply \emph{no matter} what procedure is used to optimize the cut parameters over the training set---optimal or suboptimal, automated or manual. We prove these guarantees by analyzing how the B\&C tree varies as a function of the cut parameters on any IP. By bounding the ``intrinsic complexity'' of this function, we are able to provide our sample complexity bounds.

\begin{figure}[h]
    \centering
    \includegraphics[scale=0.40]{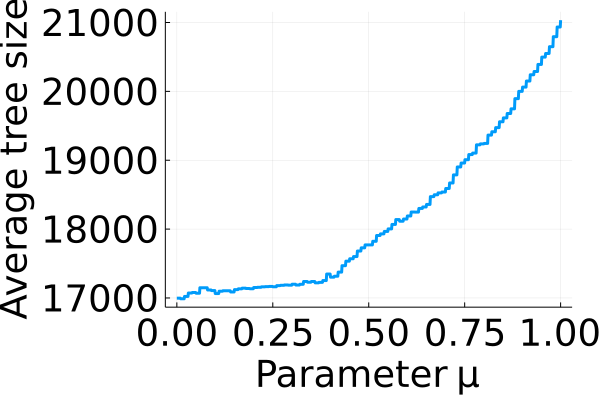}
    \caption{Facility location with 40 locations and 40 clients. Samples generated by perturbing a base facility location IP.}
    \label{fig:40_40}
\end{figure}

\begin{figure}[h]
    \centering
    \includegraphics[scale=0.40]{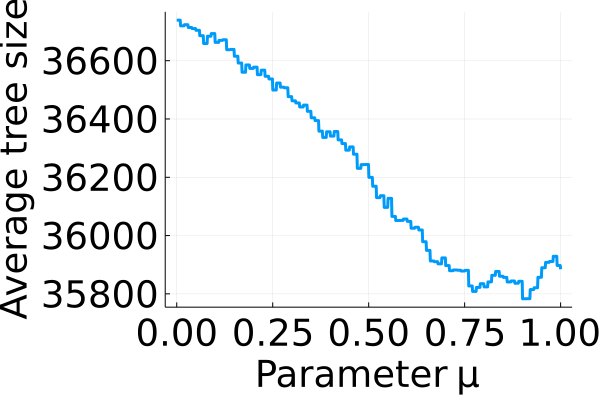}
    \caption{Facility location with 80 locations, 80 clients, and random Euclidean distance costs.}
    \label{fig:80_80}
\end{figure}
Figures~\ref{fig:40_40} and~\ref{fig:80_80} illustrate the need for distribution-dependent policies for choosing cutting planes. We plot the average number of nodes expanded by B\&C as a function of a parameter $\mu$ that controls its cut-selection policy, as we detail in Appendix~\ref{app:experiments}. In each figure, we draw a training set of facility location integer programs from two different distributions. In Figure~\ref{fig:40_40}, we define the distribution by starting with a uniformly random facility location instance and perturbing its costs. In Figure~\ref{fig:80_80}, the costs are more structured: the facilities are located along a line and the clients have uniformly random locations. In Figure~\ref{fig:40_40}, a smaller value of $\mu$ leads to small search trees, but in Figure~\ref{fig:80_80}, a larger value of $\mu$ is preferable. These figures illustrate that tuning cut parameters according to the instance distribution at hand can have a large impact on the performance of B\&C, and that for one instance distribution, the best parameters for cut evaluation can be very different---in fact opposite---than the optimal parameters for another instance distribution.

The key challenge we face in developing a theory for cutting planes is that a cut added at the root remains in the LP relaxations stored in each node all the way to the leaves, thereby impacting the LP guidance that B\&C uses to search throughout the whole tree. Tiny changes to any cut can thus completely change the entire course of B\&C. At its core, our analysis therefore involves understanding an intricate interplay between the continuous and discrete components of our problem. The first, continuous component requires us to characterize how the solution the LP relaxation changes as a function of its constraints. This optimal solution will move continuously through space until it jumps from one vertex of the LP tableau to another. We then use this characterization to analyze how the B\&C tree---a discrete, combinatorial object---varies as a function of its LP guidance.

\subsection{Our contributions}\label{sec:contrib}

Our first main contribution (Section~\ref{sec:lp_sensitivity}) addresses a fundamental question: how does an LP's solution change when new constraints are added? As the constraints vary, the solution will jump from vertex to vertex of the LP polytope. We prove that one can partition the set of all possible constraint vectors into a finite number of regions such that within any one region, the LP's solution has a clean closed form. Moreover, we prove that the boundaries defining this partition have a specific form, defined by degree-2 polynomials.

We build on this result to prove our second main contribution (Section~\ref{sec:cuts}), which analyzes how the entire B\&C search tree changes as a function of the cuts added at the root. To prove this result, we analyze how every aspect of B\&C---the variables branched on, the nodes selected to expand, and the nodes fathomed---changes as a function of the LP relaxations that are computed throughout the search tree. We prove that the set of all possible cuts can be partitioned into a finite number of regions such that within any one region, B\&C builds the exact same search tree.

This result allows us to prove sample complexity bounds for learning high-performing cutting planes from the class of \emph{Gomory mixed integer (GMI) cuts}, our third main contribution (Section~\ref{sec:sample}). GMI cuts are one of the most important families of cutting planes in the field of integer programming. Introduced by~\citet{Gomory60:Algorithm}, they dominate most other families of cutting planes~\citep{Cornuejols01:Elementary}, and are perhaps most directly responsible for the realization that a branch-and-cut framework is necessary for the speeds now achievable by modern IP solvers~\citep{Balas96:Gomory}. A historical account of these cuts is provided by~\citet{Cornuejols07:Revival}. The structural results from Section~\ref{sec:cuts} allow us to understand the ``intrinsic complexity'' of B\&C's performance as a function of the GMI cuts it uses. We quantify this notion of intrinsic complexity using \emph{pseudo-dimension}~\citep{Pollard84:Convergence}, which then implies a sample complexity bound.

\subsection{Related research}

\paragraph{Learning to cut.} This paper helps develop a theory of generalization for cutting plane selection. This line of inquiry began with a paper by \citet{Balcan21:Sample}, who studied Chv\'{a}tal-Gomory cuts for (pure) integer programs (IPs). Unlike that work, which exploited the fact that there are only finitely many distinct Chv\'{a}tal-Gomory cuts for a given IP, our analysis of GMI cuts is far more involved.

The main distinction between our analysis in this paper and the techniques used in previous papers on generalization guarantees for integer programming~\citep{Balcan18:Learning,Balcan21:Sample,Balcan21:Improved} can be summarized as follows. Let $\vec{\mu}$ be a (potentially multidimensional) parameter controlling some aspect of the IP solver (e.g. a mixture parameter between branching rules or a cutting-plane parameter). In previous works, as $\vec{\mu}$ varied, there were only a finite number of states each node of branch-and-cut could be in. For example, in the case of branching/variable selection, $\vec{\mu}$ controls the additional branching constraint added to the IP at any given node of the search tree. There are only finitely many possible branching constraints, so there are only finitely many possible ``child" IPs induced by $\vec{\mu}$. Similarly, if $\vec{\mu}$ represents the parameterization for Chv\'{a}tal-Gomory cuts~\citep{Chvatal73:Edmonds,Gomory58:Outline}, since~\citet{Balcan21:Sample} showed that there are only finitely many distinct Chv\'{a}tal-Gomory cuts for a given IP, as $\vec{\mu}$ varies, there are only finitely many possible child IPs induced by $\vec{\mu}$ at any stage of the search tree. However, in many settings, this property does not hold. For example if $\vec{\mu} = (\vec{\alpha}, \beta)$ controls the normal vector and offset of an additional feasible constraint $\vec{\alpha}^T\vec{x}\le\beta$, there are infinitely many possible IPs corresponding to the choice of $(\vec{\alpha}, \beta)$. Similarly, if $\vec{\mu}$ controls the parameterization of a GMI cut, there are infinitely many IPs corresponding to the choice of $\vec{\mu}$ (unlike Chv\'{a}tal-Gomory cuts). In this paper, we develop a new structural understanding of B\&C that is significantly more involved than the structural results in prior work.

This paper ties in to a broader line of research that provides sample complexity bounds for algorithm configuration~\citep[e.g.,][]{Gupta17:PAC,Balcan21:How}. A chapter by~\citet{Balcan20:Data} provides a comprehensive survey.

There have also been several papers that study how to use machine learning for cut selection from an applied perspective~\citep{Tang20:Reinforcement,Huang22:Learning}. In contrast, the goal of this paper is to provide theoretical guarantees.

\paragraph{Sensitivity analysis of integer and linear programs.} A related line of research studied the \emph{sensitivity} of LPs, and to a lesser extent IPs, to changes in their parameters. \citet{Mangasarian87:Lipschitz} and \citet{Li93:Sharp}, for example, show that the optimal solution to an LP is a Lipschitz function of the right-hand-side of its constraints but not of its objective. \citet{Cook86:Sensitivity} study how the set of optimal solutions to an IP changes as the objective function varies and the right-hand-side of the constraints varies. This paper fits in to this line of research as we study how the solution to an LP varies as new rows are added. This function is not Lipschitz, but we show that it is well-structured.
\section{Notation and branch-and-cut background}\label{sec:background}

\paragraph{Integer and linear programs.} An \emph{integer program} (IP) is defined by an objective vector $\vec{c} \in \R^n$, a constraint matrix $A \in \Z^{m \times n}$, and a constraint vector $\vec{b} \in \Z^m$, with the form \begin{equation}\max\{\vec{c}^T\vec{x} : A \vec{x} \leq \vec{b}, \vec{x} \geq \vec{0},  \vec{x} \in \Z^n\}.\label{eq:IP}\end{equation} The \emph{linear programming (LP) relaxation} is formed by removing the integrality constraints: \begin{equation}\max\{\vec{c}^T\vec{x} : A \vec{x} \leq \vec{b}, \vec{x} \geq \vec{0}\}.\label{eq:LP}\end{equation}

We denote the optimal solution to~\eqref{eq:IP} by $\vec{x}^*_{\IP}$. We denote the optimal solution to ~\eqref{eq:LP} by $\vec{x}^*_{\LP}$ and its objective value by $z^*_{\LP} = \vec{c}^T\vec{x}^*_{\LP}$. If $\sigma$ is a set of constraints, we let $\vec{x}^*_{\LP}(\sigma)$ denote the LP optimum of~\eqref{eq:LP} subject to these additional constraints (similarly define $z^*_{\LP}(\sigma)$ and $\vec{x}^*_{\IP}(\sigma)$).

\paragraph{Polyhedra and polytopes.} A set $\cP\subseteq\R^n$ is a \emph{polyhedron} if there exists an integer $m$, $A\in\R^{m\times n}$, and $\vec{b}\in\R^m$ such that $\cP = \{\vec{x}\in\R^n : A\vec{x}\le\vec{b}\}$. $\cP$ is a \emph{rational polyhedron} if there exists $A\in\Z^{m\times n}$ and $\vec{b}\in\Z^{m}$ such that $\cP = \{\vec{x}\in\R^n : A\vec{x}\le\vec{b}\}$. A bounded polyhedron is called a \emph{polytope}. The feasible regions of all IPs considered in this paper are assumed to be rational polytopes \footnote{This assumption is not a restrictive one. The Minkowski-Weyl theorem states that any polyhedron can be decomposed as the sum of a polytope and its recession cone. All results in this paper can be derived for rational polyhedra by considering the corresponding polytope in the Minkowski-Weyl decomposition.}. Let $\cP = \{\vec{x}\in\R^n : \vec{a}^i\vec{x}\le b_i, i \in M\}$ be a nonempty polyhedron. For any $I\subseteq M$, the set $F_I:= \{\vec{x}\in\R^n : \vec{a}^i\vec{x}=b_i, i\in I, \vec{a}^i\vec{x}\le b_i, i \in M\setminus I\}$ is a face of $\cP$. Conversely, if $F$ is a nonempty face of $\cP$, then $F = F_I$ for some $I\subseteq M$. Given a set of constraints $\sigma$, let $\cP(\sigma)$ denote the polyhedron that is the intersection of $\cP$ with all inequalities in $\sigma$.

\paragraph{Cutting planes.}
A \emph{cutting plane} is a linear constraint $\vec{\alpha}^T\vec{x} \leq \beta$. Let $\cP$ be the feasible region of the LP relaxation in Equation~\eqref{eq:LP} and $\cP_{\I} = \cP \cap \Z^n$ be the feasible set of the IP in Equation~\eqref{eq:IP}. A cutting plane is \emph{valid} if it is satisfied by every integer-feasible point: $\vec{\alpha}^T\vec{x} \leq \beta$ for all $\vec{x} \in \cP_{\I}$. A valid cut \emph{separates} a point $\vec{x} \in \cP \setminus \cP_{\I}$ if $\vec{\alpha}^T\vec{x} > \beta.$
We interchangeably refer to a cut by its parameters $(\vec{\alpha},\beta)\in\R^{n+1}$ and the halfspace $\vec{\alpha}^T\vec{x}\le\beta$ in $\R^n$ it defines.

An important family of cuts that we study in this paper is the set of \emph{Gomory mixed integer (GMI) cuts}.
\begin{definition}[Gomory mixed integer cut]\label{def:GMI}
Suppose the feasible region of the IP is in equality form $A\vec{x} = \vec{b}$, $\vec{x}\ge\vec{0}$ (which can be achieved by adding slack variables). For $\vec{u}\in\R^m$, let $f_i$ denote the fractional part of $(\vec{u}^TA)_i$ and let $f_0$ denote the fractional part of $\vec{u}^T\vec{b}$. That is, $(\vec{u}^TA)_i = (\lfloor\vec{u}^TA\rfloor)_i + f_i$ and $\vec{u}^T\vec{b} = \lfloor\vec{u}^T\vec{b} \rfloor + f_0$. The \emph{Gomory mixed integer (GMI) cut} parameterized by $\vec{u}$ is given by $$\sum_{i : f_i\le f_0}f_ix_i + \frac{f_0}{1-f_0}\sum_{i : f_i>f_0}(1-f_i)x_i \ge f_0.$$
\end{definition}

\paragraph{Branch-and-cut.} We provide a high-level overview of branch-and-cut (B\&C) and refer the reader to the textbook by~\citet{Nemhauser99:Integer} for more details. Given an IP, B\&C searches through the IP's feasible region by building a binary search tree. B\&C solves the LP relaxation of the input IP and then adds any number of cutting planes. It stores this information at the root of its binary search tree. Let $\vec{x}_{\LP}^* = (\vec{x}_{\LP}^*[1], \dots, \vec{x}_{\LP}^*[n])$ be the solution to the LP relaxation with the addition of the cutting planes. B\&C next uses a \emph{variable selection policy} to choose a variable $x_i$ to branch on. This means that it splits the IP's feasible region in two: one set where $x_i \leq \lfloor \vec{x}_{\LP}^*[i] \rfloor$ and the other where $x_i \geq \lceil\vec{x}_{\LP}^*[i]\rceil$. The left child of the root now corresponds to the IP with a feasible region defined by the first subset and the right child likewise corresponds to the second subset. B\&C then chooses a leaf using a \emph{node selection policy} and recurses, adding any number of cutting planes, branching on a variable, and so on. B\&C \emph{fathoms} a node---which means that it will never branch on that node---if 1) the LP relaxation at the node is infeasible, 2) the optimal solution to the LP relaxation is integral, or 3) the optimal solution to the LP relaxation is no better than the best integral solution found thus far. Eventually, B\&C will fathom every leaf, and it can be verified that it has found the globally optimal integral solution. We assume there
is a bound $\kappa$ on the size of the tree we allow B\&C to build before we terminate, as is common in
prior research~\citep{Hutter09:Paramils,Kleinberg17:Efficiency,Kleinberg19:Procrastinating,Balcan18:Learning,Balcan21:Improved,Balcan21:Sample}.

Every step of B\&C---including node and variable selection and the choice of whether or not to fathom---depends crucially on guidance from LP relaxations.
To give an example, this is true of the \emph{product scoring rule}~\citep{Achterberg07:Constraint}, a popular variable selection policy that our results apply to.
\begin{definition}\label{def:product}
Let $\vec{x}_{\LP}^*$ be the solution to the LP relaxation at a node and $z_{\LP}^* = \vec{c}^T\vec{x}_{\LP}^*$. The \emph{product scoring rule} branches on the variable $i \in [n]$ that maximizes: $\max\{z_{\LP}^* - z_{\LP}^*(x_i \leq \lfloor \vec{x}_{\LP}^*[i] \rfloor), 10^{-6}\}
\cdot \max\{z_{\LP}^* - z_{\LP}^*(x_i \geq \lceil \vec{x}_{\LP}^*[i] \rceil), 10^{-6}\}$.
\end{definition}
The tighter the LP relaxation, the more valuable the LP guidance, highlighting the importance of cutting planes.

\paragraph{Polynomial arrangements in Euclidean space.} Let $p\in\R[y_1,\ldots, y_k]$ be a polynomial of degree at most $d$. The polynomial $p$ partitions $\R^k$ into connected components that belong to either $\R^k\setminus\{(y_1,\ldots, y_k) : p(y_1,\ldots, y_k) = 0\}$ or $\{(y_1,\ldots, y_k) : p(y_1,\ldots, y_k) = 0\}$. When we discuss the connected components of $\R^k$ induced by $p$, we include connected components in both these sets. We make this distinction because previous work on sample complexity for data-driven algorithm design oftentimes only needed to consider the connected components of the former set. The number of connected components in both sets is $O(d^k)$~\citep{Warren68:Lower,Milnor64:Betti,Thom65:Homologie}.
\section{Linear programming sensitivity}\label{sec:lp_sensitivity}

Our main result in this section characterizes how an LP's optimal solution is affected by the addition of one or more new constraints. In particular, fixing an LP with $m$ constraints and $n$ variables, if $\vec{x}^*_{\LP}(\cut)\in\R^n$ denotes the new LP optimum when the constraint $\cut$ is added, we pin down a precise characterization of $\vec{x}^*_{\LP}(\cut)$ as a function of $\vec{\alpha}$ and $\beta$. We show that $\vec{x}^*_{\LP}(\cut)$ has a piece-wise closed form: there are surfaces partitioning $\R^{n+1}$ such that within each connected component induced by these surfaces, $\vec{x}^*_{\LP}(\cut)$ has a closed form. While the geometric intuition used to establish this piece-wise structure relies on the basic property that optimal solutions to LPs are achieved at vertices, the surfaces defining the regions are perhaps surprisingly nonlinear: they are defined by multivariate degree-$2$ polynomials in $\vec{\alpha},\beta$. In Appendix~\ref{apx:2d_example} we illustrate these surfaces for an example two-variable LP.

There are two main steps of our proof: (1) tracking the set of edges of the LP polytope intersected by the new constraint, and once that set of edges is fixed, (2) tracking which edge yields the vertex with the highest objective value.

Let $M = [m]$ denote the set of $m$ constraints. For $E\subseteq M$, let $A_E\in\R^{|E|\times n}$ and $\vec{b}_E\in\R^{|E|}$ denote the restrictions of $A$ and $\vec{b}$ to $E$. For $\vec{\alpha}\in\R^n$, $\beta\in\R$, and $E\subseteq M$ with $|E| = n-1$, let $A_{E, \vec{\alpha}}\in\R^{n\times n}$ denote the matrix obtained by adding row vector $\vec{\alpha}$ to $A_E$ and let $A^i_{E, \vec{\alpha}, \beta}\in\R^{n\times n}$ be the matrix $A_{E, \vec{\alpha}}$ with the $i$th column replaced by $(\vec{b}_E,\beta)^T$.

\begin{theorem}\label{thm:lp_main}
Let $(\vec{c}, A, \vec{b})$ be an LP and let $\vec{x}^*_{\LP}$ denote the optimal solution.
There is a set of at most $m^n$ hyperplanes and at most $m^{2n}$ degree-$2$ polynomial hypersurfaces partitioning $\R^{n+1}$ into connected components such that for each component $C$, one of the following holds: either (1) $\vec{x}^*_{\LP}(\vec{\alpha}^T\vec{x}\le\beta) = \vec{x}^*_{\LP}$ or (2) there is a set of constraints $E\subseteq M$ with $|E| = n-1$ such that $$\vec{x}^*_{\LP}(\vec{\alpha}^T\vec{x}\le\beta) =  \left(\frac{\det(A_{E, \vec{\alpha}, \beta}^1)}{\det(A_{E, \vec{\alpha}})},\ldots,\frac{\det(A_{E, \vec{\alpha}, \beta}^n)}{\det(A_{E, \vec{\alpha}})}\right)$$ for all $(\vec{\alpha}, \beta)\in C$.
\end{theorem}

\begin{proof} First, if $\vec{\alpha}^T\vec{x}\le\beta$ does not separate $\vec{x}^*_{\LP}$, then $\vec{x}^*_{\LP}(\cut) = \vec{x}^*_{\LP}$. The set of all such cuts is the halfspace in $\R^{n+1}$ given by $\{(\vec{\alpha},\beta)\in\R^{n+1} : \vec{\alpha}^T\vec{x}^*_{\LP}\le\beta\}.$ All other cuts separate $\vec{x}^*_{\LP}$ and thus pass through $\cP = \{\vec{x}\in\R^n : A\vec{x}\le\vec{b}, \vec{x}\ge\vec{0}\}$, and the new LP optimum is achieved at a vertex created by the cut. We consider the new vertices formed by the cut, which lie on edges (faces of dimension $1$) of $\cP$. Letting $M$ denote the set of $m$ constraints that define $\cP$, each edge $e$ of $\cP$ can be identified with a subset $E \subset M$ of size $n-1$ such that the edge is precisely the set of all points $\vec{x}$ such that \begin{align*} \vec{a}_i^T\vec{x} = b_i\qquad &\forall\, i\in E \\ \vec{a}_i^T\vec{x} \le b_i\qquad &\forall\, i\in M\setminus E,\end{align*} where $\vec{a}_i$ is the $i$th row of $A$. Let $A_E\in\R^{n-1\times n}$ denote the restriction of $A$ to only the rows in $E$, and let $\vec{b}_E\in\R^{|E|}$ denote the entries of $\vec{b}$ corresponding to constraints in $E$. Drop the inequality constraints defining the edge, so the equality constraints define a line in $\R^n$. The intersection of the cut $\vec{\alpha}^T\vec{x}\le\beta$ and this line is precisely the solution to the system of $n$ linear equations in $n$ variables: $A_E\vec{x} =\vec{b}_E, \vec{\alpha}^T\vec{x}=\beta$. By Cramer's rule, the (unique) solution $\vec{x}=(x_1,\ldots,x_n)$ to this system is given by $x_i = \frac{\det(A_{E, \vec{\alpha}, \beta}^i)}{\det(A_{E, \vec{\alpha}})}.$ To ensure that the intersection point indeed lies on the edge of the polytope, we simply stipulate that it satisfies the inequality constraints in $M\setminus E$. That is, \begin{equation}\sum_{j=1}^na_{ij}\cdot \frac{\det(A_{E, \vec{\alpha}, \beta}^j)}{\det(A_{E, \vec{\alpha}})} \le b_i\label{eq:1edge}\end{equation}for every $i\in M\setminus E$ (note that if $\vec{\alpha}, \beta$ satisfy any of these constraints, it must be that $\det(A_{E, \vec{\alpha}})\neq 0$, which guarantees that $A_E\vec{x}=\vec{b}_E, \vec{\alpha}^T\vec{x}=\beta$ indeed has a unique solution). Multiplying through by $\det(A_{E, \vec{\alpha}})$ shows that this constraint is a halfspace in $\R^{n+1}$, since $\det(A_{E, \vec{\alpha}})$ and $\det(A^i_{E, \vec{\alpha}, \beta})$ are both linear in $\vec{\alpha}$ and $\beta$. The collection of all the hyperplanes defining the boundaries of these halfspaces over all edges of $\cP$ induces a partition of $\R^{n+1}$ into connected components such that for all $(\vec{\alpha},\beta)$ within a given connected component, the (nonempty) set of edges of $\cP$ that the hyperplane $\vec{\alpha}^T\vec{x}=\beta$ intersects is invariant.

Now, consider a single connected component, denoted by $C$ for brevity. Let $e_1,\ldots, e_k$ denote the edges intersected by cuts in $C$, and let $E_1,\ldots, E_k\subset M$ denote the sets of constraints that are binding at each of these edges, respectively. For each pair $e_p, e_q$, consider the surface \begin{equation}\sum_{i=1}^n c_i\cdot\frac{\det(A_{E_p, \vec{\alpha}, \beta}^i)}{\det(A_{E_p, \vec{\alpha}})} = \sum_{i=1}^n c_i\cdot\frac{\det(A_{E_q, \vec{\alpha}, \beta}^i)}{\det(A_{E_q, \vec{\alpha}})}.\label{eq:2edge}\end{equation} Clearing the (nonzero) denominators shows this is a degree-$2$ polynomial hypersurface in $\vec{\alpha}, \beta$ in $\R^{n+1}$. This hypersurface is the set of all $(\vec{\alpha}, \beta)$ for which the LP objective value achieved at the vertex on edge $e_p$ is equal to the LP objective value achieved at the vertex on edge $e_q$. The collection of these surfaces for each $p, q$ partitions $C$ into further connected components. Within each of these connected components, the edge containing the vertex that maximizes the objective is invariant. If this edge corresponds to binding constraints $E$, $\vec{x}^*_{\LP}(\vec{\alpha}^T\vec{x}\le\beta)$ has the closed form $\vec{x}^*_{\LP}(\vec{\alpha}^T\vec{x}\le\beta)[i] =  \frac{\det(A_{E, \vec{\alpha}, \beta}^i)}{\det(A_{E, \vec{\alpha}})}$ for all $(\vec{\alpha},\beta)$ within this component. We now count the number of surfaces used to obtain our decomposition. $\cP$ has at most $\binom{m}{n-1}\le m^{n-1}$ edges, and for each edge $E$ we first considered at most $|M\setminus E|\le m$ hyperplanes representing decision boundaries for cuts intersecting that edge (Equation~\eqref{eq:1edge}), for a total of at most $m^{n}$ hyperplanes. We then considered a degree-$2$ polynomial hypersurface for every pair of edges (Equation~\eqref{eq:2edge}), of which there are at most $\binom{m^n}{2}\le m^{2n}$.
\end{proof}

In Appendix~\ref{apx:lp_sensitivity_multiple_cuts}, we generalize Theorem~\ref{thm:lp_main} to understand $\vec{x}^*_{\LP}$ as a function of any $K$ constraints. In this case, we show that the piecewise structure is given by degree-$2K$ multivariate polynomials.
\section{Structure and sensitivity of branch-and-cut}\label{sec:cuts}


We now use Theorem~\ref{thm:lp_main} to answer a fundamental question about B\&C: how does the B\&C tree change when cuts are added at the root? Said another way, what is the structure of the B\&C tree as a function of the set of cuts? We prove that the set of all possible cuts can be partitioned into a finite number of regions where by employing cuts from any one region, the B\&C tree remains exactly the same. Moreover, we prove that the boundaries between regions are defined by constant-degree polynomials. As in the previous section, we focus on a single cut added to the root of the B\&C tree. We provide an extension to multiple cuts in Appendix~\ref{apx:multi_b&c_sensitivity}.

We outline the main steps of our analysis: 
\begin{enumerate}
    \setlength\itemsep{0em}

    \item In Lemma~\ref{lem:closed_form} we use Theorem~\ref{thm:lp_main} to understand how the LP optimum at any node in the B\&C tree behaves as a function of cuts added at the root. 
    \item In Lemma~\ref{lem:product_branching}, we analyze how the branching decisions of B\&C are impacted by variations in the cuts. 
    \item In Lemma~\ref{lemma:integrality}, we analyze how cuts affect which nodes are fathomed due to the integrality of the LP relaxation. 
    \item In Theorem~\ref{thm:tree_invariant}, we analyze how the LP estimates based on cuts can lead to pruning nodes of the B\&C tree, which gives us a complete description of when two cutting planes lead to the same B\&C tree.
\end{enumerate}

The full proofs from this section are in Appendix~\ref{app:cuts}.

Given an IP, let $\tau = \lceil\max_{\vec{x}\in\cP}\norm{\vec{x}}_{\infty}\rceil$ be the maximum magnitude coordinate of any LP-feasible solution, rounded up. The set of all possible branching constraints is contained in $\cB\cC := \{\vec{x}[i]\le \ell, \vec{x}[i]\ge \ell\}_{0\le\ell\le\tau, i\in [n]}$ which is a set of size $2n(\tau+1)$. Na\"{i}vely, there are at most $2^{2n(\tau+1)}$ subsets of branching constraints, but the following observation allows us to greatly reduce the number of sets we consider.

\begin{lemma}\label{lemma:reduced}
Fix an IP $(\vec{c}, A, \vec{b})$. Define an equivalence relation on pairs of branching-constraint sets $\sigma_1,\sigma_2\subseteq\cB\cC$, by $\sigma_1\sim\sigma_2\iff\vec{x}^*_{\LP}(\cut, \sigma_1) = \vec{x}^*_{\LP}(\cut, \sigma_2)$ for all possible cutting planes $\cut$. The number of equivalence classes of $\sim$ is at most $\tau^{3n}$.
\end{lemma}

By Cramer's rule, $\tau\le |\det(\widetilde{A})|$, where $\widetilde{A}$ is any square submatrix of $A$. This is at most $a^n n^{n/2}$ by Hadamard's inequality, where $a$ is the maximum absolute value of any entry of $A$. However, $\tau$ can be much smaller in various cases. For example, if $A$ contains even one row with only positive entries, then $\tau\le\norm{\vec{b}}_{\infty}$.

We will use the following notation in the remainder of this section. Let $A_{\sigma}$ and $\vec{b}_{\sigma}$ denote the augmented constraint matrix and vector when the constraints in $\sigma \subseteq \cB\cC$ are added. For $E\subseteq M\cup\sigma$, let $A_{E, \sigma}\in\R^{|E|\times n}$ and $\vec{b}_E\in\R^{|E|}$ denote the restrictions of $A_{\sigma}$ and $\vec{b}_{\sigma}$ to $E$. For $\vec{\alpha}\in\R^n, \beta\in\R$ and $E\subseteq M\cup\sigma$ with $|E| = n-1$, let $A_{E, \vec{\alpha}, \sigma}\in\R^{n\times n}$ denote the matrix obtained by adding row vector $\vec{\alpha}$ to $A_{E, \sigma}$ and let $A^i_{E, \vec{\alpha}, \beta, \sigma}\in\R^{n\times n}$ be the matrix $A_{E, \vec{\alpha}, \sigma}$ with the $i$th column replaced by $(\vec{b}_{E, \sigma}, \beta)^T$.

\begin{lemma}\label{lem:closed_form}
For any LP $(\vec{c}, A, \vec{b})$, there are at most $(m+2n)^{n}\tau^{3n}$ hyperplanes and at most $(m+2n)^{2n}\tau^{3n}$ degree-$2$ polynomial hypersurfaces partitioning $\R^{n+1}$ into connected components such that for each component $C$ and every $\sigma\subset\cB\cC$, either: (1) $\vec{x}^*_{\LP}(\cut, \sigma) = \vec{x}^*_{\LP}(\sigma)$ and $z^*_{\LP}(\cut, \sigma)=z^*_{\LP}(\sigma)$, or (2) there is a set of constraints $E\subseteq M\cup\sigma$ with $|E| = n-1$ such that $\vec{x}^*_{\LP}(\cut, \sigma)[i]
=  \frac{\det(A_{E, \vec{\alpha}, \beta, \sigma}^i)}{\det(A_{E, \vec{\alpha}, \sigma})}$ for all $(\vec{\alpha}, \beta)\in C$.
\end{lemma}

\begin{proof}[Proof sketch]
The same reasoning in the proof of Theorem~\ref{thm:lp_main} yields a partition with the desired properties.
\end{proof}

Next, we refine the decomposition obtained in Lemma~\ref{lem:closed_form} so that the branching constraints added at each step of B\&C are invariant within a region. Our results apply to the product scoring rule (Def.~\ref{def:product}), which is used, for example, by the leading open-source solver SCIP~\citep{Bestuzheva21:SCIP}.

\begin{lemma}\label{lem:product_branching}
There are at most $3(m+2n)^n\tau^{3n}$ hyperplanes, $3(m+2n)^{3n}\tau^{4n}$ degree-2 polynomial hypersurfaces, and $(m+2n)^{6n}\tau^{4n}$ degree-5 polynomial hypersurfaces partitioning $\R^{n+1}$ into connected components such that within each component, the branching constraints used at every step of B\&C are invariant.
\end{lemma}

\begin{proof}[Proof sketch]
Fix a connected component $C$ in the decomposition established in Lemma~\ref{lem:closed_form}. Then, for each $\sigma$, either $\vec{x}^*_{\LP}(\cut, \sigma) = \vec{x}^*_{\LP}(\sigma)$ or there exists $E\subseteq M\cup\sigma$ such that $\vec{x}^*_{\LP}(\vec{\alpha}^T\vec{x}\le\beta,\sigma)[i] = \frac{\det(A^i_{E, \vec{\alpha}, \beta, \sigma})}{\det(A_{E, \vec{\alpha}, \sigma})}$ for all $(\vec{\alpha}, \beta)\in C$ and all $i \in [n]$. Now, if we are at a stage in the branch-and-cut tree where $\sigma$ is the list of branching constraints added so far, and the $i$th variable is being branched on next, the two constraints generated are $x_i\le\lf\vec{x}^*_{\LP}(\vec{\alpha}^T\vec{x}\le\beta,\sigma)[i]\rf$ and $x_i\ge\left\lceil \vec{x}^*_{\LP}(\vec{\alpha}^T\vec{x}\le\beta,\sigma)[i]\right\rceil$, respectively. If $C$ is a component where $\vec{x}^*_{\LP}(\cut, \sigma) = \vec{x}^*_{\LP}(\sigma)$, then there is nothing more to do, since the branching constraints at that point are trivially invariant over $(\vec{\alpha}, \beta)\in C$. Otherwise, in order to further decompose $C$ such that the right-hand-side of these constraints are invariant for every $\sigma$ and every $i = 1,\ldots, n$, we add the two decision boundaries given by $$k\le\frac{\det(A^i_{E, \vec{\alpha}, \beta, \sigma})}{\det(A_{E, \vec{\alpha}, \sigma})} \le k+1$$ for every $i$, $\sigma$, and every integer $k = 0,\ldots, \tau-1$. This ensures that within every connected component of $C$ induced by these boundaries (hyperplanes), $\lfloor\vec{x}^*_{\LP}(\vec{\alpha}^T\vec{x}\le\beta,\sigma)[i]\rfloor$ and $\lceil \vec{x}^*_{\LP}(\vec{\alpha}^T\vec{x}\le\beta,\sigma)[i]\rceil$ are invariant. A careful analysis of the definition of the product scoring rule provides the appropriate refinement of this partition.
\end{proof}

We now move to the most critical phase of branch-and-cut: deciding when to fathom a node. One reason a node might be fathomed is if the LP relaxation of the IP at that node has an integral solution. We derive conditions that ensure that nearby cuts have the same effect on the integrality of the original IP at any node in the search tree. Recall that $\cP_{\I} = \cP\cap\Z^n$ is the set of integer points in $\cP$. Let $\cV\subseteq\R^{n+1}$ denote the set of all valid cuts for the input IP $(\vec{c}, A, \vec{b})$. The set $\cV$ is a polyhedron since it can be expressed as $$\cV = \bigcap_{\vec{\overline{x}}\in\cP_{\I}}\{(\vec{\alpha}, \beta)\in\R^{n+1} : \vec{\alpha}^T\vec{\overline{x}}\le\beta\},$$ and $\cP_{\I}$ is finite as $\cP$ is bounded. For cuts outside $\cV$, we assume the B\&C tree takes some special form denoting an invalid cut. Our goal now is to decompose $\cV$ into connected components such that $\mathbf{1}\left[\vec{x}_{\LP}^*(\vec{\alpha}^T\vec{x}\le\beta, \sigma)\in\Z^n\right]$ is invariant for all $(\vec{\alpha}, \beta)$ in each component.

\begin{lemma}\label{lemma:integrality}
For any IP $(\vec{c}, A, \vec{b})$, there are at most $3(m+2n)^n\tau^{4n}$ hyperplanes, $3(m+2n)^{3n}\tau^{4n}$ degree-$2$ polynomial hypersurfaces, and $(m+2n)^{6n}\tau^{4n}$ degree-$5$ polynomial hypersurfaces partitioning $\R^{n+1}$ into connected components such that for each component $C$ and each $\sigma\subseteq\cB\cC$, $\mathbf{1}\left[\vec{x}^*_{\LP}\left(\cut, \sigma\right)\in\Z^n\right]$ is invariant for all $(\vec{\alpha}, \beta)\in C$.
\end{lemma}

\begin{proof}[Proof sketch]
Fix a connected component $C$ in the decomposition that includes the facets defining $\cV$ and the surfaces obtained in Lemma~\ref{lem:product_branching}. For all $\sigma$, $\vec{x}_{\I}\in\cP_{\I}$, and $i\in [n]$, consider the surface \begin{equation}\label{eq:integer_point}\vec{x}^*_{\LP}(\cut, \sigma)[i]=\vec{x}_{\I}[i].\end{equation} By Lemma~\ref{lem:closed_form}, this surface is a hyperplane. Clearly, within any connected component of $C$ induced by these hyperplanes, for every $\sigma$ and $\vec{x}_{\I}\in\cP_{\I}$, $\mathbf{1}[\vec{x}^*_{\LP}(\cut, \sigma) = \vec{x}_{\I}]$ is invariant. Finally, if $\vec{x}^*_{\LP}(\cut, \sigma)\in\Z^n$ for some cut $\cut$ within a given connected component, $\vec{x}^*_{\LP}(\cut, \sigma) = \vec{x}_{\I}$ for some $\vec{x}_{\I}\in\cP_{\I}(\sigma)\subseteq\cP_{\I}$, which means that $\vec{x}^*_{\LP}(\cut, \sigma) = \vec{x}_{\I}\in\Z^n$ \emph{for all} cuts $\cut$ in that connected component.
\end{proof}

Suppose for a moment that a node is fathomed by B\&C if and only if either the LP at that node is infeasible, or the LP optimal solution is integral---that is, the ``bounding" of B\&C is suppressed. In this case, the partition of $\R^{n+1}$ obtained in Lemma~\ref{lemma:integrality} guarantees that the tree built by branch-and-cut is invariant within each connected component. Indeed, since the branching constraints at every node are invariant, and for every $\sigma$ the integrality of $\vec{x}^*_{\LP}(\cut, \sigma)$ is invariant, the (bounding-suppressed) B\&C tree (and the order in which it is built) is invariant within each connected component in our decomposition. Equipped with this observation, we now analyze the full behavior of B\&C.

\begin{theorem}\label{thm:tree_invariant}
Given an IP $(\vec{c}, A, \vec{b})$, there is a set of at most $O(14^n(m+2n)^{3n^2}\tau^{5n^2})$ polynomial hypersurfaces of degree $\le 5$ partitioning $\R^{n+1}$ into connected components such that the branch-and-cut tree built after adding the cut $\cut$ at the root is invariant over all $(\vec{\alpha}, \beta)$ within a given component.
\end{theorem}

\begin{proof}[Proof sketch] Fix a connected component $C$ in the decomposition induced by the set of hyperplanes and degree-$2$ hypersurfaces established in Lemma~\ref{lemma:integrality}. Let \begin{equation}\label{eq:node_list}Q_1,\ldots, Q_{i_1}, I_1, Q_{i_1+1},\ldots,Q_{i_2},I_2,Q_{i_2+1},\ldots\end{equation} denote the nodes of the tree branch-and-cut creates, in order of exploration, under the assumption that a node is pruned if and only if either the LP at that node is infeasible or the LP optimal solution is integral (so the ``bounding'' of branch-and-bound is suppressed). Here, a node is identified by the list $\sigma$ of branching constraints added to the input IP. Nodes labeled by $Q$ are either infeasible or have fractional LP optimal solutions. Nodes labeled by $I$ have integral LP optimal solutions and are candidates for the incumbent integral solution at the point they are encountered. (The nodes are functions of $\vec{\alpha}$ and $\beta$, as are the indices $i_1, i_2,\ldots$.) By Lemma~\ref{lemma:integrality} and the observation following it, this ordered list of nodes is invariant over all $(\vec{\alpha}, \beta)\in C$.

Now, given an node index $\ell$, let $I(\ell)$ denote the incumbent node with the highest objective value encountered up until the $\ell$th node searched by B\&C, and let $z(I(\ell))$ denote its objective value. For each node $Q_{\ell}$, let $\sigma_{\ell}$ denote the branching constraints added to arrive at node $Q_{\ell}$. The hyperplane \begin{equation}\label{eq:prune}z^*_{\LP}(\vec{\alpha}^T\vec{x}\le\beta,\sigma_{\ell}) = z(I(\ell))\end{equation} (which is a hyperplane due to Lemma~\ref{lem:closed_form}) partitions $C$ into two subregions. In one subregion, $z^*_{\LP}(\vec{\alpha}^T\vec{x}\le\beta,\sigma_{\ell}) \le z(I(\ell))$, that is, the objective value of the LP optimal solution is no greater than the objective value of the current incumbent integer solution, and so the subtree rooted at $Q_{\ell}$ is pruned. In the other subregion, $z^*_{\LP}(\vec{\alpha}^T\vec{x}\le\beta,\sigma_{\ell}) > z(I(\ell))$, and $Q_{\ell}$ is branched on further. Therefore, within each connected component of $C$ induced by all hyperplanes given by Equation~\ref{eq:prune} for all $\ell$, the set of node within the list~(\ref{eq:node_list}) that are pruned is invariant. Combined with the surfaces established in Lemma~\ref{lemma:integrality}, these hyperplanes partition $\R^{n+1}$ into connected components such that as $(\vec{\alpha}, \beta)$ varies within a given component, the tree built by branch-and-cut is invariant. \end{proof}

\section{Sample complexity bounds for B\&C}\label{sec:sample}

In this section, we show how the results from the previous section can be used to provide sample complexity bounds for configuring B\&C.
Our results will apply to families of cuts parameterized by vectors $\vec{u}$ from a set $\cU$, such as the family of GMI cuts from Definition~\ref{def:GMI}. We assume there is an unknown, application-specific distribution $\dist$ over IPs. The learner receives a \emph{training set} $\sample \sim \dist^N$ of $N$ IPs sampled from this distribution. A \emph{sample complexity guarantee} bounds the number of samples $N$ sufficient to ensure that for any parameter setting $\vec{u}\in \cU$, the B\&C tree size on average over the training set $\sample$ is close to the expected B\&C tree size. More formally, let $g_{\vec{u}}(\vec{c}, A, \vec{b})$ be the size of the tree B\&C builds given the input $(\vec{c}, A, \vec{b})$ after applying the cut defined by $\vec{u}$ at the root. Given $\epsilon > 0$ and $\delta \in (0,1)$, a sample complexity guarantee bounds the number of samples $N$ sufficient to ensure that with probability $1-\delta$ over the draw $\sample \sim \dist^N$, for every parameter setting $\vec{u} \in \cU$, \begin{equation}
   \Bigg| \frac{1}{N} \sum_{(\vec{c}, A, \vec{b}) \in \sample}g_{\vec{u}}(\vec{c}, A, \vec{b}) - \E\left[g_{\vec{u}}(\vec{c}, A, \vec{b})\right]\Bigg| \leq \epsilon.
\label{eq:generalization}\end{equation}

To derive our sample complexity guarantee, we use the notion of \emph{pseudo-dimension}~\citep{Pollard84:Convergence}. Let $\cG = \{g_{\vec{u}} : \vec{u} \in \cU\}$. The \emph{pseudo-dimension of $\cG$}, denoted $\pdim(\cG)$, is the largest integer $N$ for
which there exist $N$ IPs $(\vec{c}_1, A_1, \vec{b}_1), \dots, (\vec{c}_N, A_N, \vec{b}_N)$ and $N$ thresholds $r_1, \dots, r_N \in \R$ such that for every binary vector $(\sigma_1, \dots, \sigma_N) \in \{0,1\}^N$, there exists $g_{\vec{u}} \in \cG$ such that $g_{\vec{u}}(\vec{c}_i, A_i, \vec{b}_i) \geq r_i$ if and only if $\sigma_i = 1$. The number of samples sufficient to ensure that Equation~\eqref{eq:generalization} holds is $N = O(\frac{\kappa^2}{\epsilon^2}(\pdim(\cG) + \log \frac{1}{\delta}))$~\citep{Pollard84:Convergence}. Equivalently, for a given number of samples $N$, the left-hand-side of Equation~\eqref{eq:generalization} can be bounded by $\kappa\sqrt{\frac{1}{N}(\pdim(\cG) + \log\frac{1}{\delta})}$.

So far, $\vec{\alpha}, \beta$ are parameters that do not depend on the input instance $\vec{c}, A, \vec{b}$. Suppose now that they do: $\vec{\alpha}, \beta$ are functions of $\vec{c}, A, \vec{b}$ and a parameter vector $\vec{u}$ (as they are for GMI cuts). Despite the structure established in the previous section, if $\vec{\alpha},\beta$ can depend on $(\vec{c}, A, \vec{b})$ in arbitrary ways, one cannot even hope for a finite sample complexity, illustrated by the following impossibility result. The full proofs of all results from this section are in Appendix~\ref{app:sample}.

\begin{theorem}\label{thm:impossible}
There exist functions $\vec{\alpha}_{\vec{c}, A, \vec{b}}:\cU\to\R^n$ and $\beta_{\vec{c}, A, \vec{b}}:\cU\to\R$ such that $$\pdim\left(\left\{g_{\vec{u}} : \vec{u} \in \cU\right\}\right)=\infty,$$ where $\cU$ is any set with $|\cU| = |\R|$. 
\end{theorem}

However, in the case of GMI cuts, we show that the cutting plane coefficients parameterized by $\vec{u}$ are highly structured. Combining this structure with our analysis of B\&C allows us to derive polynomial sample complexity bounds.

\begin{lemma}\label{lemma:gmi_hyperplanes}
Consider the family of GMI cuts parameterized by $\vec{u}\in [-U, U]^m$. There is a set of at most $O(nU^2\norm{A}_1\norm{\vec{b}}_1)$ hyperplanes partitioning $[-U, U]^m$ into connected components such that $\lfloor\vec{u}^T\vec{a}_i\rfloor$, $\lfloor\vec{u}^T\vec{b}\rfloor$, and $\mathbf{1}[f_i\le f_0]$ are invariant, for every $i$, within each component.
\end{lemma}

\begin{proof}[Proof sketch]
We have $f_i = \vec{u}^T\vec{a}_i - \lfloor\vec{u}^T\vec{a}_i\rfloor$, $f_0 = \vec{u}^T\vec{b}-\lfloor\vec{u}^T\vec{b}\rfloor$, and since $\vec{u}\in [-U, U]^m$, $\lfloor\vec{u}^T\vec{a}_i\rfloor\in [-U\norm{\vec{a}_i}_1,U\norm{\vec{a}_i}_1]$ and $\lfloor\vec{u}^T\vec{b}\rfloor\in [-U\norm{\vec{b}}_1, U\norm{\vec{b}}_1].$ For all $i$, $k_i\in[-U\norm{\vec{a}_i}_1,U\norm{\vec{a}_i}_1]\cap\Z$ and $k_0\in[-U\norm{\vec{b}}_1,U\norm{\vec{b}}_1]\cap\Z$, hyperplanes define the two regions \[\lf\vec{u}^T\vec{a}_i\rf = k_i \iff k_i\le \vec{u}^T\vec{a}_i < k_i + 1\] and the hyperplanes defining the two halfspaces \[\lf\vec{u}^T\vec{b}\rf = k_0 \iff k_0\le \vec{u}^T\vec{b} < k_0 + 1.\] In addition,  for each $i$, consider the hyperplane \begin{equation}\label{eq:gmi_form_main}\vec{u}^T\vec{a}_i - k_i = \vec{u}^T\vec{b}-k_0.\end{equation} Within any connected component of $\R^m$ determined by these hyperplanes, $\lfloor\vec{u}^T\vec{a}_i\rfloor$ and $\lfloor\vec{u}^T\vec{b}\rfloor$ are constant. Also, $\mathbf{1}[f_i\le f_0]$ is invariant within each component, since if $\lfloor\vec{u}^T\vec{a}_i\rfloor = k_i$ and $\lfloor\vec{u}^T\vec{b}\rfloor = k_0$, $f_i\le f_0\iff \vec{u}^T\vec{a}_i - k_i \le \vec{u}^T\vec{b}-k_0,$ which is the hyperplane from Equation~\ref{eq:gmi_form_main}.
The lemma follows by counting the hyperplanes.
\end{proof}

Let $\vec{\alpha}:[-U,U]^m\to\R^n$ denote the function taking GMI cut parameters $\vec{u}$ to the corresponding vector of coefficients determining the resulting cutting plane, and let $\beta:[-U,U]^m\to\R$ denote the offset of the resulting cutting plane. So (after multiplying through by $1-f_0$), $$\vec{\alpha}(\vec{u})[i] = \begin{cases}f_i(1-f_0) & \text{if }f_i\le f_0 \\ f_0(1-f_i) & \text{if }f_i > f_0 \end{cases}$$ and $\beta(\vec{u}) = f_0(1-f_0)$ (of course $f_0$ and each $f_i$ are functions of $\vec{u}$, but we suppress this dependence for readability).

The next lemma allows us to transfer the polynomial partition of $\R^{n+1}$ from Theorem~\ref{thm:tree_invariant} to a polynomial partition of $[-U, U]^m$, incurring only a factor $2$ increase in degree.

\begin{lemma}\label{lemma:gmi_polynomial}
Let $p\in\R[y_1,\ldots,y_{n+1}]$ be a polynomial of degree $d$. Let $D\subseteq [-U, U]^m$ be a connected component from Lemma~\ref{lemma:gmi_hyperplanes}. Define $q:D\to\R$ by $q(\vec{u}) = p(\vec{\alpha}(\vec{u}), \beta(\vec{u}))$. Then $q$ is a polynomial in $\vec{u}$ of degree $2d$.
\end{lemma}

\begin{proof}By Lemma~\ref{lemma:gmi_hyperplanes}, there are integers $k_0, k_i$ for $i\in [n]$ such that $\lfloor \vec{u}^T\vec{a}_i\rfloor = k_i$ and $\lfloor\vec{u}^T\vec{b}\rfloor = k_0$ for all $\vec{u}\in D$. Also, the set $S = \{i : f_i\le f_0\}$ is fixed over all $\vec{u}\in D$.

A degree-$d$ polynomial $p$ in variables $y_1,\ldots, y_{n+1}$ can be written as $\sum_{T\sqsubseteq [n+1], |T|\le d}\lambda_T\prod_{i\in T}y_i$ for some coefficients $\lambda_T\in\R$, where $T\sqsubseteq [n+1]$ means that $T$ is a multiset of $[n+1]$. Evaluating at $(\vec{\alpha}(\vec{u}),\beta(\vec{u}))$, we get $$\sum_{|T|\le d}\lambda_T\prod_{\substack{i\in T\cap S \\ i\neq n+1}} f_i(1-f_0)\prod_{\substack{i\in T\setminus S \\ i\neq n+1}}f_0(1-f_i)\prod_{\substack{i\in T \\ i = n+1}}f_0(1-f_0).$$ Now, $f_i = \vec{u}^T\vec{a}_i-k_i$ and $f_0 = \vec{u}^T\vec{b}-k_0$ are linear in $\vec{u}$. The sum is over all multisets of size at most $d$, so each monomial consists of the product of at most $d$ degree-$2$ terms of the form $f_i(1-f_0)$, $f_0(1-f_i)$, or $f_0(1-f_0)$. Thus, $\deg(q)\le 2d$, as desired.
\end{proof}

Applying Lemma~\ref{lemma:gmi_polynomial} to every polynomial hypersurface in the partition of $\R^{n+1}$ established in Theorem~\ref{thm:tree_invariant} yields our main structural result for GMI cuts.

\begin{lemma}\label{lemma:gmi_tree_invariant} Consider the family of GMI cuts parameterized by $\vec{u}\in [-U, U]^m$. For any IP $(\vec{c}, A, \vec{b})$, there are at most $O(nU^2\norm{A}_1\norm{\vec{b}}_1)$ hyperplanes and $2^{O(n^2)}(m+2n)^{O(n^3)}\tau^{O(n^3)}$ degree-$10$ polynomial hypersurfaces partitioning $[-U, U]^{m}$ into connected components such that the B\&C tree built after adding the GMI cut defined by $\vec{u}$ is invariant over all $\vec{u}$ within a single component.
\end{lemma}

Bounding the pseudo-dimension of the class of tree-size functions $\{g_{\vec{u}}:\vec{u}\in [-U, U]^m\}$ is a direct application of the main theorem of~\citet{Balcan21:How} along with standard results bounding the VC dimension of polynomial boundaries~\citep{Anthony09:Neural}.

\begin{theorem}\label{thm:gmi_pdim}
The pseudo-dimension of the class of tree-size functions $\{g_{\vec{u}} : \vec{u}\in[-U, U]^m\}$ on the domain of IPs with $\norm{A}_1\le a$ and $\norm{\vec{b}}_{1}\le b$ is $$O\left(m\log(abU) + mn^3\log(m+n) + mn^3\log\tau\right).$$ 
\end{theorem}

We generalize the analysis of this section to multiple GMI cuts at the root of the B\&C tree in Appendix~\ref{app:sample}. The analysis there is more involved since GMI cuts can be applied in sequence, re-solving the LP relaxation after each cut. In particular, GMI cuts applied in sequence have one more parameter than the next, so the hyperplane defined by each GMI cut depends (polynomially) on the parameters defining all GMI cuts before it.  We show that if $K$ GMI cuts are sequentially applied at the root, the resulting partition of the parameter space is induced by polynomials of degree $O(K^2)$.
\section{Conclusions}
In this paper, we investigated fundamental questions about linear and integer programs: given an integer program, how many possible branch-and-cut trees are there if one or more additional feasible constraints can be added? Even more specifically, what is the structure of the branch-and-cut tree as a function of a set of additional constraints? Through a detailed geometric and combinatorial analysis of how additional constraints affect the LP relaxation's optimal solution, we showed that the branch-and-cut tree is piecewise constant and precisely bounded the number of pieces.
We showed that the structural understandings that we developed could be used to prove sample complexity bounds for configuring branch-and-cut.

\subsection*{Acknowledgements}
This material is based on work supported by the National Science Foundation under grants CCF-1733556, CCF-1910321, IIS-1901403, and SES-1919453, the ARO under award W911NF2010081, the Defense Advanced Research Projects Agency under cooperative agreement HR00112020003, a Simons Investigator Award, an AWS Machine Learning Research Award, an Amazon Research Award, a Bloomberg Research Grant, and a Microsoft Research Faculty Fellowship.

\bibliographystyle{plainnat}
\bibliography{dairefs}

\newpage

\appendix
\section{Further details about plots}\label{app:experiments}

The version of the \emph{facility location} problem we study involves a set of locations $J$ and a set of clients $C$. Facilities are to be constructed at some subset of the locations, and the clients in $C$ are served by these facilities. Each location $j\in J$ has a cost $f_j$ of being the site of a facility, and a cost $s_{c, j}$ of serving client $c\in C$. Finally, each location $j$ has a capacity $\kappa_{j}$ which is a limit on the number of clients $j$ can serve. The goal of the facility location problem is to arrive at a feasible set of locations for facilities and a feasible assignment of clients to these locations that minimizes the overall cost incurred.

The facility location problem can be formulated as the following $0,1$ IP: $$\begin{array}{ll} \text{minimize} & \displaystyle\sum_{j\in J}f_j x_j + \sum_{j\in J}\sum_{c\in C}s_{c,j}y_{c,j}\\
\text{subject to} & \displaystyle \sum_{j\in J}y_{c, j} = 1\hfill\forall\,c\in C\\
& \displaystyle \sum_{c\in C}y_{c, j}\le \kappa_{j}x_j\hfill\forall\,j\in J\\
& y_{c, j}\in\{0,1\}\hfill\qquad\forall\, c\in C, j\in J \\
& x_j \in \{0,1\}\hfill\forall\, j\in J
\end{array}$$

We consider the following two distributions over facility location IPs.

\paragraph{First distribution}
Facility location IPs are generated by perturbing the costs and capacities of a base facility location IP. We generated the base IP with $40$ locations and $40$ clients by choosing the location costs and client-location costs uniformly at random from $[0, 100]$ and the capacities uniformly at random from $\{0,\ldots, 39\}$. To sample from the distribution, we perturb this base IP by adding independent Gaussian noise with mean $0$ and standard deviation $10$ to the cost of each location, the cost of each client-location pair, and the capacity of each location.

\paragraph{Second distribution}
Facility location IPs are generated by placing $80$ evenly-spaced locations along the line segment connecting the points $(0, 1/2)$ and $(1, 1/2)$ in the Cartesian plane. The location costs are all uniformly set to $1$. Then, $80$ clients are placed uniformly at random in the unit square $[0,1]^2$. The cost $s_{c, j}$ of serving client $c$ from location $j$ is the distance between $j$ and $c$. Location capacities are chosen uniformly at random from $\{0,\ldots,43\}$.

In our experiments, we add five cuts at the root of the B\&C tree. These five cuts come from the set of Chv\'{a}tal-Gomory and Gomory mixed integer cuts derived from the optimal simplex tableau of the LP relaxation. The five cuts added are chosen to maximize a weighting of cutting-plane scores: \begin{equation}\label{eq:score}\mu\cdot\score_1 + (1-\mu)\cdot\score_2.\end{equation} $\score_1$ is the \emph{parallelism} of a cut, which intuitively measures the angle formed by the objective vector and the normal vector of the cutting plane---promoting cutting planes that are nearly parallel with the objective direction. $\score_2$ is the \emph{efficacy}, or depth, of a cut, which measures the perpendicular distance from the LP optimum to the cut---promoting cutting planes that are ``deeper", as measured with respect to the LP optimum. More details about these scoring rules can be found in~\citet{Balcan21:Sample} and references therein. Given an IP, for each $\mu\in [0,1]$ (discretized at steps of $0.01$) we choose the five cuts among the set of Chv\'{a}tal-Gomory and Gomory mixed integer cuts that maximize~\eqref{eq:score}. Figures~\ref{fig:40_40} and~\ref{fig:80_80} display the average tree size over 1000 samples drawn from the respective distribution for each value of $\mu$ used to choose cuts at the root. We ran our experiments using the C API of IBM ILOG CPLEX 20.1.0, with default cut generation disabled.

\section{Omitted results and proofs from Section~\ref{sec:lp_sensitivity}}

\subsection{Example in two dimensions}\label{apx:2d_example}

Consider the LP $$\max\{x+y: x\le 1, y\ge 0, y\le x\}.$$ The optimum is at $(x^*, y^*) = (1, 1)$. Consider adding an additional constraint $\alpha_1 x + \alpha_2 y \le 1$. Let $h$ denote the hyperplane $\alpha_1 x + \alpha_2 y=1$. We derive a description of the set of parameters $(\alpha_1, \alpha_2)$ such that $h$ intersects the hyperplanes $x= 1$ and $y= x$. The intersection of $h$ and $x= 1$ is given by $$(x, y) = \left(1, \frac{1-\alpha_1}{\alpha_2}\right),$$ which exists if and only if $\alpha_2\neq 0$. This intersection point is in the LP feasible region if and only if $0\le \frac{1-\alpha_1}{\alpha_2}\le 1$ (which additionally ensures that $\alpha_2\neq 0$). Similarly, $h$ intersects $y= x$ at $$(x, y) = \left(\frac{1}{\alpha_1+\alpha_2}, \frac{1}{\alpha_1+\alpha_2}\right),$$ which exists if and only if $\alpha_1+\alpha_2\neq 0$. This intersection point is in the LP feasible region if and only if $0\le\frac{1}{\alpha_1+\alpha_2}\le 1$. Now, we put down an ``indifference" curve in $(\alpha_1, \alpha_2)$-space that represents the set of $(\alpha_1, \alpha_2)$ such that the value of the objective achieved at the two aforementioned intersection points is equal. This surface is given by $$\frac{2}{\alpha_1 + \alpha_2} = 1 + \frac{1-\alpha_1}{\alpha_2}.$$ Since $\alpha_1 + \alpha_2\neq 0$ and $\alpha_2\neq 0$ (for the relevant $\alpha_1, \alpha_2$ in consideration), this is equivalent to $\alpha_1^2 - \alpha_2^2 - \alpha_1 + \alpha_2 = 0,$ which is a degree-$2$ curve in $\alpha_1, \alpha_2$. The left-hand-side can be factored to write this as $(\alpha_1-\alpha_2)(\alpha_1+\alpha_2-1)=0$. Therefore, this curve is given by the two lines $\alpha_1=\alpha_2$ and $\alpha_1+\alpha_2=1$. Figure~\ref{fig:cut_regions} illustrates the resulting partition of $(\alpha_1,\alpha_2)$-space.

\begin{figure}[t]
    \centering
    \includegraphics[scale=0.5]{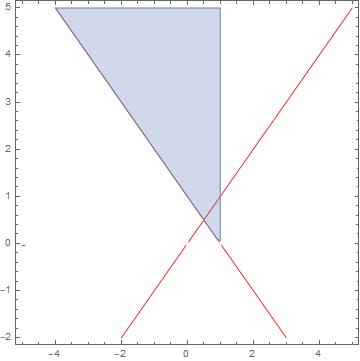}
    \caption{Decomposition of the parameter space: the blue region contains the set of $(\alpha_1, \alpha_2)$ such that the constraint intersects the feasible region at $x = 1$ and $x = y$. The red lines consist of all $(\alpha_1, \alpha_2)$ such that the objective value is equal at these intersection points. The red lines partition the blue region into two components: one where the new optimum is achieved at the intersection of $h$ and $x = y$, and one where the new optimum is achieved at the intersection of $h$ and $x = 1$.}
    \label{fig:cut_regions}
\end{figure}

It turns out that when $n=2$ the indifference curve can always be factored into a product of linear terms. Let the objective of the LP be $(c_1, c_2)$, and let $s_1x + s_2y = u_1$ and $t_1x + t_2y = v_1$ be two intersecting edges of the LP feasible region. Let $\alpha_1 x + \alpha_2 y = \beta$ be an additional constraint. The intersection points of this constraint with the two lines, if they exist, are given by $$\left(\frac{s_2\beta-u\alpha_2}{s_2\alpha_1-s_1\alpha_2}, \frac{s_1\beta-u\alpha_1}{s_1\alpha_2-s_2\alpha_1}\right)\text{ and } \left(\frac{t_2\beta-v\alpha_2}{t_2\alpha_1-t_1\alpha_2},\frac{t_2\beta-v\alpha_1}{t_1\alpha_2-t_2\alpha_1}\right).$$ The indifference surface is thus given by $$c_1\frac{s_2\beta-u\alpha_2}{s_2\alpha_1-s_1\alpha_2}+c_2 \frac{s_1\beta-u\alpha_1}{s_1\alpha_2-s_2\alpha_1} = c_1\frac{t_2\beta-v\alpha_2}{t_2\alpha_1-t_1\alpha_2} + c_2\frac{t_2\beta-v\alpha_1}{t_1\alpha_2-t_2\alpha_1}.$$ For $\alpha_1, \alpha_2$ such that $s_2\alpha_1-s_1\alpha_2\neq 0$ and $t_2\alpha_1-t_1\alpha_2\neq 0$, clearing denominators and some manipulation yields $$(c_1\alpha_2-c_2\alpha_1)((ut_1-vs_1)\alpha_2 - (ut_2-vs_2)\alpha_1 + (s_2t_2-t_1s_2)\beta) = 0.$$ This curve consists of the two planes $c_1\alpha_2-c_2\alpha_1=0$ and $(ut_1-vs_1)\alpha_2 - (ut_2-vs_2)\alpha_1 + (s_2t_2-t_1s_2)\beta=0$.

This is however not true if $n>2$. For example, consider an LP in three variables $x, y, z$ with the constraints $x+y\le 1, x+z\le 1, x\le 1, z\le 1$. Writing out the indifference surface (assuming the objective is $\vec{c} = (1, 1, 1)^T$) for the vertex on the intersection of $\{x+y=1, x=1\}$ and the vertex on $\{x+z=1, z=1\}$ yields $$\alpha_1\alpha_2-\alpha_2\beta-\alpha_3^2+\alpha_3\beta=0.$$ Setting $\beta=1$, we can plot the resulting surface in $\alpha_1,\alpha_2,\alpha_3$ (Figure~\ref{fig:nonlinear_surface}).

\begin{figure}[t]
    \centering
    \includegraphics[scale=0.5]{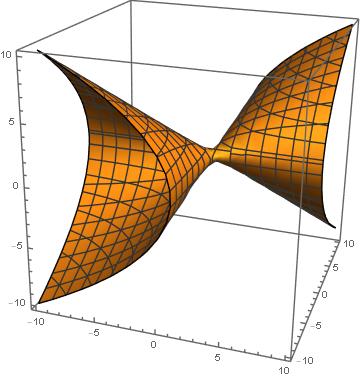}
    \caption{Indifference surface for two edges of the feasible region of an LP in three variables.}
    \label{fig:nonlinear_surface}
\end{figure}

\subsection{Linear programming sensitivity for multiple constraints}\label{apx:lp_sensitivity_multiple_cuts}

\begin{lemma}\label{lemma:multi_closed_form}
Let $(\vec{c}, A, \vec{b})$ be an LP and let $M$ denote the set of its $m$ constraints. Let $\vec{x}^*_{\LP}$ and $z^*_{\LP}$ denote the optimal solution and its objective value, respectively. For $F\subseteq M$, let $A_F\in\R^{|F|\times n}$ and $\vec{b}_F\in\R^{|F|}$ denote the restrictions of $A$ and $\vec{b}$ to $F$. For $k\le n$, $\vec{\alpha}_1,\ldots,\vec{\alpha}_k\in\R^n$, $\beta_1,\ldots,\beta_k\in\R$, and $F\subseteq M$ with $|F| = n-k$, let $A_{F, \vec{\alpha}_1,\ldots,\vec{\alpha}_k}\in\R^{n\times n}$ denote the matrix obtained by adding row vectors $\vec{\alpha}_1,\ldots,\vec{\alpha}_k$ to $A_F$ and let $A^i_{F, \vec{\alpha}_1, \beta_1,\ldots,\vec{\alpha}_k,\beta_k}\in\R^{n\times n}$ be the matrix $A_{F, \vec{\alpha}_1,\ldots,\vec{\alpha}_k}\in\R^{n\times n}$ with the $i$th column replaced by $\begin{bmatrix}\vec{b}_F & \beta_1 & \cdots & \beta_k\end{bmatrix}^T$. There is a set of at most $K$ hyperplanes, $nK^nm^n$ degree-$K$ polynomial hypersurfaces, and $nK^nm^{2n}$ degree-$2K$ polynomial hypersurfaces partitioning $\R^{K(n+1)}$ into connected components such that for each component $C$, one of the following holds: either (1) $\vec{x}^*_{\LP}(\cut[1],\ldots,\cut[K]) = \vec{x}^*_{\LP}$, or (2) there is a subset of cuts indexed by $\ell_1,\ldots, \ell_k\in[K]$ and a set of constraints $F\subseteq M$ with $|F| = n-k$ such that $$\vec{x}^*_{\LP}(\cut[1],\ldots,\cut[K]) = \left(\frac{\det(A^1_{F, \vec{\alpha}_{\ell_1},\beta_{\ell_1},\ldots, \vec{\alpha}_{\ell_k}, \beta_{\ell_k}})}{\det(A_{F,\vec{\alpha}_{\ell_1},\ldots,\vec{\alpha}_{\ell_k}})},\ldots, \frac{\det(A^n_{F, \vec{\alpha}_{\ell_1},\beta_{\ell_1},\ldots, \vec{\alpha}_{\ell_k}, \beta_{\ell_k}})}{\det(A_{F,\vec{\alpha}_{\ell_1},\ldots,\vec{\alpha}_{\ell_k}})}\right),$$ for all $(\vec{\alpha}_1, \beta_1,\ldots, \vec{\alpha}_K, \beta_K)\in C$.
\end{lemma}

\begin{proof} First, if none of $\cut[1],\ldots, \cut[K]$ separate $\vec{x}^*_{\LP}$, then $\vec{x}^*_{\LP}(\cut[1],\ldots,\cut[K]) = \vec{x}^*_{\LP}$ and $z^*_{\LP}(\cut[1],\ldots,\cut[K]) = z^*_{\LP}$. The set of all such cuts is given by the intersection of halfspaces in $\R^{K(n+1)}$ given by \begin{equation}\label{eq:separation}\bigcap_{j=1}^K\left\{(\vec{\alpha}_1,\beta_1,\ldots, \vec{\alpha}_k, \beta_k)\in\R^{K(n+1)} : \vec{\alpha}_j^T\vec{x}^*_{\LP}\le\beta_j\right\}.\end{equation} All other vectors of $K$ cuts contain at least one cut that separates $\vec{x}^*_{\LP}$, and those cuts therefore pass through $\cP = \{\vec{x}\in\R^n : A\vec{x}\le\vec{b}, \vec{x}\ge\vec{0}\}$. The new LP optimum is thus achieved at a vertex created by the cuts that separate $\vec{x}^*_{\LP}$. As in the proof of Theorem~\ref{thm:lp_main}, we consider all possible new vertices formed by our set of $K$ cuts. In the case of a single cut, these new vertices necessarily were on edges of $\cP$, but now they may lie on higher dimensional faces. 

Consider a subset of $k\le n$ cuts that separate $\vec{x}^*_{\LP}$. Without loss of generality, denote these cuts by $\cut[1],\ldots,\cut[k]$. We now establish conditions for these $k$ cuts to ``jointly" form a new vertex of $\cP$. Any vertex created by these cuts must lie on a face $f$ of $\cP$ with $\dim(f) = k$ (in the case that $k = n$, the relevant face $f$ with $\dim(f) = n$ is $\cP$ itself). Letting $M$ denote the set of $m$ constraints that define $\cP$, each dimension-$k$ face $f$ of $\cP$ can be identified with a (potentially empty) subset $F\subset M$ of size $n-k$ such that $f$ is precisely the set of all points $\vec{x}$ such that \begin{align*} \vec{a}_i^T\vec{x} = b_i\qquad &\forall\, i\in F \\ \vec{a}_i^T\vec{x} \le b_i\qquad &\forall\, i\in M\setminus F,\end{align*} where $\vec{a}_i$ is the $i$th row of $A$.  
Let $A_F\in\R^{n-k\times n}$ denote the restriction of $A$ to only the rows in $F$, and let $\vec{b}_F\in\R^{n-k}$ denote the entries of $\vec{b}$ corresponding to the constraints in $F$. Consider removing the inequality constraints defining the face. The intersection of the cuts $\cut[1],\ldots,\cut[k]$ and this unbounded surface (if it exists) is precisely the solution to the system of $n$ linear equations \begin{align*} A_F\vec{x} &=\vec{b}_F \\ \vec{\alpha}_{1}^T\vec{x}&=\beta_{1} \\ &\vdots \\ \vec{\alpha}_{k}^T\vec{x}&=\beta_{k}.\end{align*} Let $A_{F,\vec{\alpha}_{1},\ldots,\vec{\alpha}_{k}}\in\R^{n\times n}$ denote the matrix obtained by adding row vectors $\vec{\alpha}_{1},\ldots,\vec{\alpha}_{k}$ to $A_F$, and let $A^i_{F, \vec{\alpha}_{1},\beta_{1},\ldots,\vec{\alpha}_{k}, \beta_{k}}\in\R^{n\times n}$ denote the matrix $A_{F,\vec{\alpha}_{1},\ldots,\vec{\alpha}_{k}}$ where the $i$th column is replaced by $$\begin{bmatrix}\vec{b}_F \\ \beta_{1} \\ \vdots \\ \beta_{k}\end{bmatrix}\in\R^n.$$ 
By Cramer's rule, the solution to this system is given by $$\vec{x} = \left(\frac{\det(A^1_{F, \vec{\alpha}_{1},\beta_{1},\ldots, \vec{\alpha}_{k}, \beta_{k}})}{\det(A_{F,\vec{\alpha}_{1},\ldots,\vec{\alpha}_{k}})},\ldots, \frac{\det(A^n_{F, \vec{\alpha}_{1},\beta_{1},\ldots, \vec{\alpha}_{k}, \beta_{k}})}{\det(A_{F,\vec{\alpha}_{1},\ldots,\vec{\alpha}_{k}})}\right),$$ and the value of the objective at this point is $$\vec{c}^T\vec{x} = \sum_{i=1}^n c_i\cdot\frac{\det(A^i_{F, \vec{\alpha}_{1},\beta_{1},\ldots, \vec{\alpha}_{k}, \beta_{k}})}{\det(A_{F,\vec{\alpha}_{1},\ldots,\vec{\alpha}_{k}})}.$$
Now, to ensure that the unique intersection point $\vec{x}$ (1) exists and (2) actually lies on $f$ (or simply lies in $\cP$, in the case that $F=\emptyset$) , we stipulate that it satisfies the inequality constraints in $M\setminus F$. That is, \begin{equation}\label{eq:multi_cut_intersection}\sum_{j=1}^n a_{ij}\frac{\det(A^1_{F, \vec{\alpha}_{1},\beta_{1},\ldots, \vec{\alpha}_{k}, \beta_{k}})}{\det(A_{F,\vec{\alpha}_{1},\ldots,\vec{\alpha}_{k}})}\le b_i\end{equation} for every $i\in M\setminus F$. If $\vec{\alpha}_{1},\beta_1\ldots, \vec{\alpha}_{k}, \beta_k$ satisfies any of these constraints, it must be that $\det(A_{F,\vec{\alpha}_{1},\ldots,\vec{\alpha}_{k}})\neq 0$, which guarantees that $A_F\vec{x}=\vec{b}_F, \vec{\alpha}^T_1\vec{x}=\beta_1,\ldots,\vec{\alpha}^T_k\vec{x}=\beta_k$ indeed has a unique solution. Now, $\det(A_{F,\vec{\alpha}_{1},\ldots,\vec{\alpha}_{k}})$ is a polynomial in $\vec{\alpha}_1,\ldots,\vec{\alpha}_k$ of degree $\le k$, since it is multilinear in each coefficient of each $\vec{\alpha}_{\ell}$, $\ell = 1,\ldots,k$. Similarly, $\det(A^1_{F, \vec{\alpha}_{1},\beta_{1},\ldots, \vec{\alpha}_{k}, \beta_{k}})$ is a polynomial in $\vec{\alpha}_1,\beta_1,\ldots,\vec{\alpha}_k,\beta_k$ of degree $\le k$, again because it is multilinear in each cut parameter. Hence, the boundary each constraint of the form given by Equation~\ref{eq:multi_cut_intersection} is a polynomial of degree at most $k$.

The collection of these polynomials for every $k$, every subset of $\{\cut[1],\ldots,\cut[K]\}$ of size $k$, and every face of $\cP$ of dimension $k$, along with the hyperplanes determining separation constraints (Equation~\ref{eq:separation}), partition $\R^{K(n+1)}$ into connected components such that for all $(\vec{\alpha}_1, \beta_1,\ldots,\vec{\alpha}_K, \beta_K)$ within a given connected component, there is a fixed subset of $K$ and a fixed set of faces of $\cP$ such that the cuts with indices in that subset intersect every face in the set at a common vertex.

Now, consider a single connected component, denoted by $C$. Let $f_1,\ldots, f_{\ell}$ denote the faces intersected by vectors of cuts in $C$, and let (without loss of generality) $1,\ldots, k$ denote the subset of cuts that intersect these faces. Let $F_1,\ldots, F_{\ell}\subset M$ denote the sets of constraints that are binding at each of these faces, respectively. For each pair $f_p, f_q$, consider the surface $$\sum_{i=1}^n c_i\cdot\frac{\det(A^i_{F_p, \vec{\alpha}_{1},\beta_{1},\ldots, \vec{\alpha}_{k}, \beta_{k}})}{\det(A_{F_p,\vec{\alpha}_{1},\ldots,\vec{\alpha}_{k}})} = \sum_{i=1}^n c_i\cdot\frac{\det(A^i_{F_q, \vec{\alpha}_{1},\beta_{1},\ldots, \vec{\alpha}_{k}, \beta_{k}})}{\det(A_{F_q,\vec{\alpha}_{1},\ldots,\vec{\alpha}_{k}})},$$ which can be equivalently written as \begin{equation}\label{eq:multi_cut_objective}\sum_{i=1}^n c_i\cdot\det(A^i_{F_p, \vec{\alpha}_{1},\beta_{1},\ldots, \vec{\alpha}_{k}, \beta_{k}})\det(A_{F_q,\vec{\alpha}_{1},\ldots,\vec{\alpha}_{k}}) = \sum_{i=1}^n c_i\cdot\det(A^i_{F_q, \vec{\alpha}_{1},\beta_{1},\ldots, \vec{\alpha}_{k}, \beta_{k}}) \det(A_{F_p,\vec{\alpha}_{1},\ldots,\vec{\alpha}_{k}}).\end{equation} This is a degree-$2k$ polynomial hypersurface in $(\vec{\alpha}_1,\beta_1,\ldots,\vec{\alpha}_K, \beta_K)\in\R^{K(n+1)}$. This hypersurface is precisely the set of all cut vectors for which the LP objective achieved at the vertex on face $f_p$ is equal to the LP objective value achieved at the vertex on face $f_q$. The collection of these surfaces for each $p, q$ partitions $C$ into further connected components. Within each of these connected components, the face containing the vertex that maximizes the objective is invariant, and the subset of cuts passing through that vertex is invariant. If $F\subseteq M$ is the set of binding constraints representing this face, and $\ell_1,\ldots, \ell_k\in[K]$ represent the subset of cuts intersecting this face, $\vec{x}^*_{\LP}(\cut[1],\ldots,\cut[K])$ and $z^*_{\LP}(\cut[1],\ldots,\cut[K])$ have the closed forms: $$\vec{x}^*_{\LP}(\cut[1],\ldots,\cut[K]) = \left(\frac{\det(A^1_{F, \vec{\alpha}_{\ell_1},\beta_{\ell_1},\ldots, \vec{\alpha}_{\ell_k}, \beta_{\ell_k}})}{\det(A_{F,\vec{\alpha}_{\ell_1},\ldots,\vec{\alpha}_{\ell_k}})},\ldots, \frac{\det(A^n_{F, \vec{\alpha}_{\ell_1},\beta_{\ell_1},\ldots, \vec{\alpha}_{\ell_k}, \beta_{\ell_k}})}{\det(A_{F,\vec{\alpha}_{\ell_1},\ldots,\vec{\alpha}_{\ell_k}})}\right),$$ and $$z^*_{\LP}(\cut[1],\ldots,\cut[K]) = \sum_{i=1}^n c_i\cdot\frac{\det(A^i_{F, \vec{\alpha}_{\ell_1},\beta_{\ell_1},\ldots, \vec{\alpha}_{\ell_k}, \beta_{\ell_k}})}{\det(A_{F,\vec{\alpha}_{\ell_1},\ldots,\vec{\alpha}_{\ell_k}})}.$$ for all $(\vec{\alpha}_1,\beta_1,\ldots,\vec{\alpha}_K, \beta_K)$ within this component. We now count the number of surfaces used to obtain our decomposition. First, we added $K$ hyperplanes encoding separation constraints for each of the $K$ cuts (Equation~\ref{eq:separation}). Then, for every subset $S\subseteq K$ of size $\le n$, and for every face $F$ of $\cP$ with $\dim(F) = |S|$, we first considered at most $|M\setminus F|\le m$ degree-$\le K$ polynomial hypersurfaces representing decision boundaries for when cuts in $S$ intersected that face (Equation~\ref{eq:multi_cut_intersection}). The number of $k$-dimensional faces of $\cP$ is at most $\binom{m}{n-k}\le m^{n-k}\le m^{n-1}$, so the total number of these hypersurfaces is at most $(\binom{K}{0}+\cdots+\binom{K}{n})m^{n}\le nK^n m^n$. Finally, we considered a degree-$2K$ polynomial hypersurface for every subset of cuts and every pair of faces with degree equal to the size of the subset, of which there are at most $nK^n\binom{m^{n}}{2}\le nK^nm^{2n}$. \end{proof}

\section{Omitted results and proofs from Section~\ref{sec:cuts}}\label{app:cuts}

\begin{proof}[Proof of Lemma~\ref{lemma:reduced}]
Consider as an example $\sigma_1 = \{\vec{x}[1]\le 1, \vec{x}[1]\le 5\}$ and $\sigma_2 = \{\vec{x}[1]\le 1\}$. We have $\vec{x}^*_{\LP}(\cut, \sigma_1) = \vec{x}^*_{\LP}(\cut, \sigma_2)$ for any cut $\cut$, because the constraint $\vec{x}[1]\le 5$ is redundant in $\sigma_1$. More generally, any $\sigma\subseteq\cB\cC$ can be reduced by preserving only the tightest $\le$ constraint and tightest $\ge$ constraint without affecting the resulting LP optimal solutions. The number of such unique reduced sets is at most $((\tau+2)^2)^n < \tau^{3n}$ (for each variable, there are $\tau+2$ possibilities for the tightest $\le$ constraint: no constraint or one of $\vec{x}[i]\le 0,\ldots,\vec{x}[i]\le\tau$, and similarly $\tau+2$ possibilities for the $\ge$ constraint).
\end{proof}

\begin{proof}[Proof of Lemma~\ref{lem:closed_form}]
We carry out the same reasoning in the proof of Theorem~\ref{thm:lp_main} for each reduced $\sigma$. The number of edges of $\cP(\sigma)$ is at most $\binom{m+|\sigma|}{n-1}\le (m+|\sigma|)^{n-1}$. For each edge $E$, we considered at most $|(M\cup\sigma)\setminus E|\le m+|\sigma|$ hyperplanes, for a total of at most $(m+|\sigma|)^n$ halfspaces. Then, we had a degree-$2$ polynomial hypersurface for every pair of edges, of which there are at most $\binom{(m+|\sigma|)^n}{2}\le (m+|\sigma|)^{2n}$. Summing over all reduced $\sigma$ (of which there are at most $\tau^{3n}$), combined with the fact that if $\sigma$ is reduced then $|\sigma| \le 2n$, we get a total of at most $(m+2n)^n\tau^{3n}$ hyperplanes and at most $(m+2n)^{2n}\tau^{3n}$ degree-$2$ hypersurfaces, as desired. \end{proof}

\begin{proof}[Proof of Lemma~\ref{lemma:integrality}]
Fix a connected component $C$ in the decomposition that includes the facets defining $\cV$ and the surfaces obtained in Lemma~\ref{lem:product_branching}. For all $\sigma\in\cB\cC$, $\vec{x}_{\I}\in\cP_{\I}$, and $i = 1,\ldots, n$, consider the surface \begin{equation}\label{eq:integer_point}\vec{x}^*_{\LP}(\cut, \sigma)[i]=\vec{x}_{\I}[i].\end{equation} This surface is a hyperplane, since by Lemma~\ref{lem:closed_form}, either $\vec{x}^*_{\LP}(\cut, \sigma)[i]=\vec{x}^*_{\LP}(\sigma)[i]$ or $\vec{x}^*_{\LP}(\cut, \sigma)[i] = \frac{\det(A^i_{E, \vec{\alpha}, \beta, \sigma})}{\det(A_{E, \vec{\alpha}, \sigma})}$, where $E\subseteq M\cup\sigma$ is the subset of constraints corresponding to $\sigma$ and $C$. Clearly, within any connected component of $C$ induced by these hyperplanes, for every $\sigma$ and $\vec{x}_{\I}\in\cP_{\I}$, $\mathbf{1}[\vec{x}^*_{\LP}(\cut, \sigma) = \vec{x}_{\I}]$ is invariant. Finally, if $\vec{x}^*_{\LP}(\cut, \sigma)\in\Z^n$ for some cut $\cut$ within a given connected component, $\vec{x}^*_{\LP}(\cut, \sigma) = \vec{x}_{\I}$ for some $\vec{x}_{\I}\in\cP_{\IH}(\sigma)\subseteq\cP_{\I}$, which means that $\vec{x}^*_{\LP}(\cut, \sigma) = \vec{x}_{\I}\in\Z^n$ \emph{for all} cuts $\cut$ in that connected component.

We now count the number of hyperplanes given by Equation~\ref{eq:integer_point}. For each $\sigma$, there are $\binom{m+|\sigma|}{n-1}\le (m+2n)^{n-1}$ binding edge constraints $E\subseteq M\cup\sigma$ defining the formula of Lemma~\ref{lem:closed_form}, and we have $n|\cP_{\I}|$ hyperplanes for each $E$. Since $\tau = \max_{\vec{x}\in\cP_{\I}}\norm{\vec{x}}_{\infty}$, $|\cP_{\I}|\le\tau^n$. So the total number of hyperplanes given by Equation~\ref{eq:integer_point} is at most $\tau^{3n}(m+2n)^{n-1}n\tau^n\le(m+2n)^n\tau^{4n}.$ The number of facets defining $\cV$ is at most $|\cP_{\IH}|\le|\cP_{\I}|\le\tau^n$. Adding these to the counts obtained in Lemma~\ref{lem:product_branching} yields the final tallies in the lemma statement.
\end{proof}

\begin{proof}[Proof of Theorem~\ref{thm:tree_invariant}] Fix a connected component $C$ in the decomposition induced by the set of hyperplanes and degree-$2$ hypersurfaces established in Lemma~\ref{lemma:integrality}. Let \begin{equation}\label{eq:node_list}Q_1,\ldots, Q_{i_1}, I_1, Q_{i_1+1},\ldots,Q_{i_2},I_2,Q_{i_2+1},\ldots\end{equation} denote the nodes of the tree branch-and-cut creates, in order of exploration, under the assumption that a node is pruned if and only if either the LP at that node is infeasible or the LP optimal solution is integral (so the ``bounding'' of branch-and-bound is suppressed). Here, a node is identified by the list $\sigma$ of branching constraints added to the input IP. Nodes labeled by $Q$ are either infeasible or have fractional LP optimal solutions. Nodes labeled by $I$ have integral LP optimal solutions and are candidates for the incumbent integral solution at the point they are encountered. (The nodes are functions of $\vec{\alpha}$ and $\beta$, as are the indices $i_1, i_2,\ldots$.) By Lemma~\ref{lemma:integrality} and the observation following it, this ordered list of nodes is invariant over all $(\vec{\alpha}, \beta)\in C$.

Now, given an node index $\ell$, let $I(\ell)$ denote the incumbent node with the highest objective value encountered up until the $\ell$th node searched by B\&C, and let $z(I(\ell))$ denote its objective value. For each node $Q_{\ell}$, let $\sigma_{\ell}$ denote the branching constraints added to arrive at node $Q_{\ell}$. The hyperplane \begin{equation}\label{eq:prune}z^*_{\LP}(\vec{\alpha}^T\vec{x}\le\beta,\sigma_{\ell}) = z(I(\ell))\end{equation} (which is a hyperplane due to Lemma~\ref{lem:closed_form}) partitions $C$ into two subregions. In one subregion, $z^*_{\LP}(\vec{\alpha}^T\vec{x}\le\beta,\sigma_{\ell}) \le z(I(\ell))$, that is, the objective value of the LP optimal solution is no greater than the objective value of the current incumbent integer solution, and so the subtree rooted at $Q_{\ell}$ is pruned. In the other subregion, $z^*_{\LP}(\vec{\alpha}^T\vec{x}\le\beta,\sigma_{\ell}) > z(I(\ell))$, and $Q_{\ell}$ is branched on further. Therefore, within each connected component of $C$ induced by all hyperplanes given by Equation~\ref{eq:prune} for all $\ell$, the set of node within the list~(\ref{eq:node_list}) that are pruned is invariant. Combined with the surfaces established in Lemma~\ref{lemma:integrality}, these hyperplanes partition $\R^{n+1}$ into connected components such that as $(\vec{\alpha}, \beta)$ varies within a given component, the tree built by branch-and-cut is invariant.

Finally, we count the total number of surfaces inducing this partition. Unlike the counting stages of the previous lemmas, we will first have to count the number of connected components induced by the surfaces established in Lemma~\ref{lemma:integrality}. This is because the ordered list of nodes explored by branch-and-cut~(\ref{eq:node_list}) can be different across each component, and the hyperplanes given by Equation~\ref{eq:prune} depend on this list. From Lemma~\ref{lemma:integrality} we have $3(m+2n)^n\tau^{4n}$ hyperplanes, $3(m+2n)^{3n}\tau^{4n}$ degree-$2$ polynomial hypersurfaces, and $(m+2n)^{6n}\tau^{4n}$ degree-$5$ polynomial hypersurfaces. To determine the connected components of $\R^{n+1}$ induced by the zero sets of these polynomials, it suffices to consider the zero set of the product of all polynomials defining these surfaces. Denote this product polynomial by $p$. The degree of the product polynomial is the sum of the degrees of $3(m+2n)^n\tau^{4n}$ degree-$1$ polynomials, $3(m+2n)^{3n}\tau^{4n}$ degree-$2$ polynomials, and $(m+2n)^{6n}\tau^{4n}$ degree-$5$ polynomials, which is at most $3(m+2n)^n\tau^{4n}+2\cdot 3(m+2n)^{3n}\tau^{4n} + 5\cdot (m+2n)^{6n}\tau^{4n} < 14(m+2n)^{3n}\tau^{4n}$. By Warren's theorem, the number of connected components of $\R^{n+1}\setminus\{(\vec{\alpha}, \beta) : p(\vec{\alpha}, \beta) = 0\}$ is $O((14(m+2n)^{3n}\tau^{4n})^{n-1})$, and by the Milnor-Thom theorem, the number of connected components of $\{(\vec{\alpha}, \beta) : p(\vec{\alpha}, \beta) = 0\}$ is $O((14(m+2n)^{3n}\tau^{4n})^{n-1})$ as well. So, the number of connected components induced by the surfaces in Lemma~\ref{lemma:integrality} is $O(14^n(m+2n)^{3n^2}\tau^{4n^2}).$ For every connected component $C$ in Lemma~\ref{lemma:integrality}, the closed form of $z^{*}_{\LP}(\cut, \sigma_{\ell})$ is already determined due to Lemma~\ref{lem:closed_form}, and so the number of hyperplanes given by Equation~\ref{eq:prune} is at most the number of possible $\sigma\subseteq\cB\cC$, which is at most $\tau^{3n}$. So across all connected components $C$, the total number of hyperplanes given by Equation~\ref{eq:prune} is $O(14^n(m+2n)^{3n^2}\tau^{5n^2}).$ Finally, adding this to the surface-counts established in Lemma~\ref{lemma:integrality} yields the lemma statement. \end{proof}

\subsection{Product scoring rule for variable selection}
Let $\sigma$ be the set of branching constraints added thus far. The product scoring rule branches on the variable $i \in [n]$ that maximizes:
\[\max\{z_{\LP}^*(\sigma) - z_{\LP}^*(x_i \leq \lfloor x_{\LP}^*(\sigma)[i] \rfloor, \sigma), \gamma\} \cdot \max\{z_{\LP}^*(\sigma) - z_{\LP}^*(x_i \geq \lceil x_{\LP}^*(\sigma)[i] \rceil, \sigma), \gamma\},\] where $\gamma=10^{-6}$.

\begin{lemma}\label{lem:branch_constraints}
 There is a set of of at most $3(m+2n)^n\tau^{3n}$ hyperplanes and $(m+2n)^{2n}\tau^{3n}$ degree-$2$ polynomial hypersurfaces partitioning $\R^{n+1}$ into connected components such that for any connected component $C$ and any $\sigma$, the set of branching constraints $\{x_i\le\lf\vec{x}^*_{\LP}(\vec{\alpha}^T\vec{x}\le\beta,\sigma)[i]\rf, x_i\ge\left\lceil \vec{x}^*_{\LP}(\vec{\alpha}^T\vec{x}\le\beta,\sigma)[i]\right\rceil \mid i \in [n]\}$ is invariant across all $(\vec{\alpha}, \beta) \in C$.
\end{lemma}

\begin{proof}
Fix a connected component $C$ in the decomposition established in Lemma~\ref{lem:closed_form}. By Lemma~\ref{lem:closed_form}, for each $\sigma$, either $\vec{x}^*_{\LP}(\cut, \sigma) = \vec{x}^*_{\LP}(\sigma)$ or there exists $E\subseteq M\cup\sigma$ such that $\vec{x}^*_{\LP}(\vec{\alpha}^T\vec{x}\le\beta,\sigma)[i] = \frac{\det(A^i_{E, \vec{\alpha}, \beta, \sigma})}{\det(A_{E, \vec{\alpha}, \sigma})}$ for all $(\vec{\alpha}, \beta)\in C$. Fix a variable $i \in [n]$, which corresponds to two branching constraints \begin{equation}x_i\le\lf\vec{x}^*_{\LP}(\vec{\alpha}^T\vec{x}\le\beta,\sigma)[i]\rf \text{ and } x_i\ge\left\lceil \vec{x}^*_{\LP}(\vec{\alpha}^T\vec{x}\le\beta,\sigma)[i]\right\rceil.\label{eq:branching}\end{equation} If $C$ is a component where $\vec{x}^*_{\LP}(\cut, \sigma) = \vec{x}^*_{\LP}(\sigma)$, then these two branching constraints are trivially invariant over $(\vec{\alpha}, \beta)\in C$. Otherwise, in order to further decompose $C$ such that the right-hand-sides of these constraints are invariant for every $\sigma$, we add the two decision boundaries given by $$k\le\frac{\det(A^i_{E, \vec{\alpha}, \beta, \sigma})}{\det(A_{E, \vec{\alpha}, \sigma})} \le k+1$$ for every $i$, $\sigma$, and every integer $k = 0,\ldots, \tau-1$, where $\tau = \max_{\vec{x}\in\cP\cap\Z^n}\norm{\vec{x}}_{\infty}$. This ensures that within every connected component of $C$ induced by these boundaries (hyperplanes), $$\lf\vec{x}^*_{\LP}(\vec{\alpha}^T\vec{x}\le\beta,\sigma)[i]\rf = \lf\frac{\det(A^i_{E, \vec{\alpha}, \beta, \sigma})}{\det(A_{E, \vec{\alpha}, \sigma})} \rf\text{ and }\left\lceil \vec{x}^*_{\LP}(\vec{\alpha}^T\vec{x}\le\beta,\sigma)[i]\right\rceil = \left\lceil \frac{\det(A^i_{E, \vec{\alpha}, \beta, \sigma})}{\det(A_{E, \vec{\alpha}, \sigma})}  \right\rceil$$ are invariant, so the branching constraints from Equation~\eqref{eq:branching} are invariant. For a fixed $\sigma$, there are two hyperplanes for every $E\subseteq M\cup\sigma$ corresponding to an edge of $\cP(\sigma)$ and $i = 1,\ldots, n$, for a total of at most $2n\binom{m+|\sigma|}{n-1}\le 2n(m+|\sigma|)^{n-1}$ hyperplanes. Summing over all reduced $\sigma$, we get a total of $2n(m+2n)^{n-1}\tau^{3n}<2(m+2n)^{n}\tau^{3n}$ hyperplanes. Adding these hyperplanes to the set of hyperplanes established in Lemma~\ref{lem:closed_form} yields the lemma statement.
\end{proof}

\begin{proof}[Proof of Lemma~\ref{lem:product_branching}]
Fix a connected component $C$ in the decomposition established in Lemma~\ref{lem:branch_constraints}. We know that for each set of branching constraints $\sigma$:
\begin{itemize}
    \item By Lemma~\ref{lem:closed_form}, either $\vec{x}^*_{\LP}(\cut, \sigma) = \vec{x}^*_{\LP}(\sigma)$ or there exists $E\subseteq M\cup\sigma$ such that $\vec{x}^*_{\LP}(\vec{\alpha}^T\vec{x}\le\beta,\sigma)[i] = \frac{\det(A^i_{E, \vec{\alpha}, \beta, \sigma})}{\det(A_{E, \vec{\alpha}, \sigma})}$ for all $(\vec{\alpha}, \beta)\in C$ and all $i \in [n]$, and
    \item The set of branching constraints $\{x_i\le\lf\vec{x}^*_{\LP}(\vec{\alpha}^T\vec{x}\le\beta,\sigma)[i]\rf, x_i\ge\left\lceil \vec{x}^*_{\LP}(\vec{\alpha}^T\vec{x}\le\beta,\sigma)[i]\right\rceil \mid i \in [n]\}$ is invariant across all $(\vec{\alpha}, \beta) \in C$.
\end{itemize}

Suppose that $\sigma$ is the list of branching constraints added so far.
For any variable $k \in [n]$, let \[\sigma_k^- = (x_k\le\lf\vec{x}^*_{\LP}(\vec{\alpha}^T\vec{x}\le\beta,\sigma)[k]\rf, \sigma) \text{ and } \sigma_k^+ = (x_k\ge\left\lceil \vec{x}^*_{\LP}(\vec{\alpha}^T\vec{x}\le\beta,\sigma)[k]\right\rceil, \sigma).\] So long as $(\vec{\alpha}, \beta) \in C$, $\sigma_k^-$ and $\sigma_k^+$ are fixed.
With this notation, we can write the product scoring rule as \[\max\{z_{\LP}^*(\vec{\alpha}^T\vec{x}\le\beta, \sigma) - z_{\LP}^*(\vec{\alpha}^T\vec{x}\le\beta, \sigma_k^-), \gamma\} \cdot \max\{z_{\LP}^*(\vec{\alpha}^T\vec{x}\le\beta, \sigma) - z_{\LP}^*(\vec{\alpha}^T\vec{x}\le\beta, \sigma_k^+), \gamma\},\] where $\gamma=10^{-6}$.

By Lemma~\ref{lem:closed_form}, we know that across all $(\vec{\alpha}, \beta) \in C$, either $z^*_{\LP}(\cut, \sigma_k^+) = z^*_{\LP}(\sigma_k^+)$ or there exists $E_k^+ \subseteq M\cup\sigma_k^+$ such that \[z^*_{\LP}\left(\vec{\alpha}^T\vec{x}\le\beta,\sigma_k^+\right) = \sum_{i = 1}^n c_i \cdot \frac{\det\left(A^i_{E_k^+, \vec{\alpha}, \beta, \sigma_k^+}\right)}{\det\left(A_{E_k^+, \vec{\alpha}, \sigma_k^+}\right)},\] and similarly for $\sigma_k^-$, defined according to some edge set $E_k^- \subseteq M\cup\sigma_k^-$.
Therefore, for each $k \in [n]$, there is a single degree-2 polynomial hypersurface partitioning $C$ into connected components such that within each connected component, either \begin{equation}z_{\LP}^*(\vec{\alpha}^T\vec{x}\le\beta, \sigma) - z_{\LP}^*(\vec{\alpha}^T\vec{x}\le\beta, \sigma_k^-) \geq \gamma\label{eq:gamma}\end{equation} or vice versa, and similarly for $\sigma_k^+$.
In particular, the former hypersurface will have one of four forms:
\begin{enumerate}
    \item $z^*_{\LP}(\sigma) - z^*_{\LP}(\sigma_k^-) \geq \gamma$, which is uniformly satisfied or not satisfied across all $(\vec{\alpha}, \beta) \in C$,
    \item $z^*_{\LP}(\sigma) - \sum_{i = 1}^n c_i \cdot \frac{\det\left(A^i_{E_k^-, \vec{\alpha}, \beta, \sigma_k^-}\right)}{\det\left(A_{E_k^-, \vec{\alpha}, \sigma_k^-}\right)} \geq \gamma$, which is a hyperplane,
    \item $\sum_{i = 1}^n c_i \cdot \frac{\det\left(A^i_{E, \vec{\alpha}, \beta, \sigma}\right)}{\det\left(A_{E, \vec{\alpha}, \sigma}\right)} - z^*_{\LP}(\sigma_k^-) \geq \gamma$, which is a hyperplane, or
    \item $\sum_{i = 1}^n c_i \left(\frac{\det\left(A^i_{E, \vec{\alpha}, \beta, \sigma}\right)}{\det\left(A_{E, \vec{\alpha}, \sigma}\right)} - \frac{\det\left(A^i_{E_k^-, \vec{\alpha}, \beta, \sigma_k^-}\right)}{\det\left(A_{E_k^+, \vec{\alpha}, \sigma_k^-}\right)}\right) \geq \gamma$, which is a degree-2 polynomial hypersurface.
\end{enumerate}
Simply said, these are all degree-2 polynomial hypersurfaces.

Within any region induced by these hypersurfaces, the comparison between any two variables $x_k$ and $x_j$ will have the form
\begin{align*}&\max\{z_{\LP}^*(\vec{\alpha}^T\vec{x}\le\beta, \sigma) - z_{\LP}^*(\vec{\alpha}^T\vec{x}\le\beta, \sigma_k^-), \gamma\} \cdot \max\{z_{\LP}^*(\vec{\alpha}^T\vec{x}\le\beta, \sigma) - z_{\LP}^*(\vec{\alpha}^T\vec{x}\le\beta, \sigma_k^+), \gamma\}\\
\geq \, &\max\{z_{\LP}^*(\vec{\alpha}^T\vec{x}\le\beta, \sigma) - z_{\LP}^*(\vec{\alpha}^T\vec{x}\le\beta, \sigma_j^-), \gamma\} \cdot \max\{z_{\LP}^*(\vec{\alpha}^T\vec{x}\le\beta, \sigma) - z_{\LP}^*(\vec{\alpha}^T\vec{x}\le\beta, \sigma_j^+), \gamma\}\end{align*} which at its most complex will equal
\begin{align}&\sum_{i = 1}^n c_i \left(\frac{\det\left(A^i_{E, \vec{\alpha}, \beta, \sigma}\right)}{\det\left(A_{E, \vec{\alpha}, \sigma}\right)} - \frac{\det\left(A^i_{E_k^-, \vec{\alpha}, \beta, \sigma_k^-}\right)}{\det\left(A_{E_k^-, \vec{\alpha}, \sigma_k^-}\right)}\right) \cdot \sum_{i = 1}^n c_i \left(\frac{\det\left(A^i_{E, \vec{\alpha}, \beta, \sigma}\right)}{\det\left(A_{E, \vec{\alpha}, \sigma}\right)} - \frac{\det\left(A^i_{E_k^+, \vec{\alpha}, \beta, \sigma_k^+}\right)}{\det\left(A_{E_k^+, \vec{\alpha}, \sigma_k^+}\right)}\right)\label{eq:var_comp}\\
\geq &\sum_{i = 1}^n c_i \left(\frac{\det\left(A^i_{E, \vec{\alpha}, \beta, \sigma}\right)}{\det\left(A_{E, \vec{\alpha}, \sigma}\right)} - \frac{\det\left(A^i_{E_j^-, \vec{\alpha}, \beta, \sigma_j^-}\right)}{\det\left(A_{E_j^-, \vec{\alpha}, \sigma_j^-}\right)}\right) \cdot \sum_{i = 1}^n c_i \left(\frac{\det\left(A^i_{E, \vec{\alpha}, \beta, \sigma}\right)}{\det\left(A_{E, \vec{\alpha}, \sigma}\right)} - \frac{\det\left(A^i_{E_j^+, \vec{\alpha}, \beta, \sigma_j^+}\right)}{\det\left(A_{E_j^+, \vec{\alpha}, \sigma_j^+}\right)}\right).\nonumber\end{align} This inequality can be written as a degree-5 polynomial hypersurface. In any region induced by these hypersurfaces, the variable that branch-and-cut branches on will be fixed.

We now count the total number of hypersurfaces. First, we count the number of degree-2 polynomial hypersurfaces from Equation~\eqref{eq:gamma}: there is a hypersurface defined by each variable $x_k$, set of branching constraints $\sigma$, cutoff $t \in [\tau]$ such that $\sigma_k^- = (x_k \leq t, \sigma)$, set $E\subseteq M\cup\sigma$ corresponding to an edge of $\cP(\sigma)$, and set $E_k^-\subseteq M\cup\sigma_k^-$ (and similarly for $\sigma_k^+$ and $E_k^+$). For a fixed $\sigma$, this amounts to $2n\tau\binom{m+|\sigma|}{n-1}\binom{m+|\sigma|+1}{n-1}\le 2n\tau(m+|\sigma|+1)^{2(n-1)}$ hypersurfaces. Summing over all $\tau^{3n}$ reduced $\sigma$, we have $2n\tau^{3n+1}(m+2n+1)^{2(n-1)}$ degree-2 polynomial hypersurfaces.

Next, we count the number of degree-5 polynomial hypersurfaces from Equation~\eqref{eq:var_comp}: there is a hypersurface defined by each pair of variables $x_k, x_j$, set of branching constraints $\sigma$, cutoffs $t_k, t_j \in [\tau]$ such that $\sigma_k^- = (x_k \leq t_k, \sigma)$ and $\sigma_j^- = (x_j \leq t_j, \sigma)$, and sets $E, E_k^-, E_k^+, E_j^-, E_j^+$ corresponding to edges of $\cP(\sigma),\cP(\sigma_k^-),\cP(\sigma_k^+), \cP(\sigma_j^-), \cP(\sigma_j^+)$. For a fixed $\sigma$, this amounts to $n^2\tau^2\binom{m+|\sigma|}{n-1}\binom{m+|\sigma|+1}{n-1}^4\le n^2\tau^2(m+|\sigma|+1)^{5(n-1)}$ hypersurfaces. Summing over all $\tau^{3n}$ reduced $\sigma$, we have $n^2\tau^{3n+2}(m+2n+1)^{5(n-1)}$ degree-5 polynomial hypersurfaces.

Adding these hypersurfaces to those from Lemma~\ref{lem:branch_constraints}, we get the lemma statement.
\end{proof}

\subsection{Extension to multiple cutting planes}\label{apx:multi_b&c_sensitivity}

We can similarly derive a multi-cut version of Lemma~\ref{lem:closed_form} that controls $\vec{x}^*_{\LP}(\cut[1],\ldots,\cut[K],\sigma)$ for any set of branching constraints. We use the following notation. Let $(\vec{c}, A, \vec{b})$ be an LP and let $M$ denote the set of its $m$ constraints. For $F\subseteq M\cup\sigma$, let $A_{F,\sigma}\in\R^{|F|\times n}$ and $\vec{b}_{F,\sigma}\in\R^{|F|}$ denote the restrictions of $A_{\sigma}$ and $\vec{b}_{\sigma}$ to $F$. For $\vec{\alpha}_1,\ldots,\vec{\alpha}_k\in\R^n$, $\beta_1,\ldots,\beta_k\in\R$, and $F\subseteq M\cup\sigma$ with $|F| = n-k$, let $A_{F, \vec{\alpha}_1,\ldots,\vec{\alpha}_k, \sigma}\in\R^{n\times n}$ denote the matrix obtained by adding row vectors $\vec{\alpha}_1,\ldots,\vec{\alpha}_k$ to $A_{F,\sigma}$ and let $A^i_{F, \vec{\alpha}_1, \beta_1,\ldots,\vec{\alpha}_k,\beta_k,\sigma}\in\R^{n\times n}$ be the matrix $A_{F, \vec{\alpha}_1,\ldots,\vec{\alpha}_k,\sigma}\in\R^{n\times n}$ with the $i$th column replaced by $\begin{bmatrix}\vec{b}_{F, \sigma} & \beta_1 & \cdots & \beta_k\end{bmatrix}^T$.

\begin{corollary}\label{cor:multi_closed_form}
Fix an IP $(\vec{c}, A,\vec{b})$. There is a set of at most $K$ hyperplanes, $nK^n(m+2n)^{n}\tau^{3n}$ degree-$K$ polynomial hypersurfaces, and $nK^n(m+2n)^{2n}\tau^{3n}$ degree-$2K$ polynomial hypersurfaces partitioning $\R^{K(n+1)}$ into connected components such that for each component $C$ and every $\sigma\subseteq\cB\cC$, one of the following holds: either (1) $\vec{x}^*_{\LP}(\cut[1],\ldots,\cut[K],\sigma) = \vec{x}^*_{\LP}(\sigma)$, or (2) there is a subset of cuts indexed by $\ell_1,\ldots, \ell_k\in[K]$ and a set of constraints $F\subseteq M\cup\sigma$ with $|F| = n-k$ such that $$\vec{x}^*_{\LP}(\cut[1],\ldots,\cut[K],\sigma) = \left(\frac{\det(A^1_{F, \vec{\alpha}_{\ell_1},\beta_{\ell_1},\ldots, \vec{\alpha}_{\ell_k}, \beta_{\ell_k},\sigma})}{\det(A_{F,\vec{\alpha}_{\ell_1},\ldots,\vec{\alpha}_{\ell_k},\sigma})},\ldots, \frac{\det(A^n_{F, \vec{\alpha}_{\ell_1},\beta_{\ell_1},\ldots, \vec{\alpha}_{\ell_k}, \beta_{\ell_k},\sigma})}{\det(A_{F,\vec{\alpha}_{\ell_1},\ldots,\vec{\alpha}_{\ell_k},\sigma})}\right),$$ for all $(\vec{\alpha}_1, \beta_1,\ldots, \vec{\alpha}_K, \beta_K)\in C$.
\end{corollary}

\begin{proof}
The exact same reasoning in the proof of Lemma~\ref{lemma:multi_closed_form} applies. We still have $K$ hyperplanes. Now, for each $\sigma$, for each subset $S\subseteq K$ with $|S|\le n$, and for every face $F$ of $\cP(\sigma)$ with $\dim(F) = |S|$, we have at most $m$ degree-$K$ polynomial hypersurfaces. The number of $k$-dimensional faces of $\cP(\sigma)$ is at most $\binom{m+|\sigma|}{n-k}\le (m+2n)^{n-1}$, so the total number of these hypersurfaces is at most $nK^n(m+2n)^n\tau^{3n}$. Finally, for every $\sigma$, we considered a degree-$2K$ polynomal hypersurfaces for every subset of cuts and every pair of faces with degree equal to the size of the subset, of which there are at most $nK^n(m+2n)^{2n}\tau^{3n}$, as desired.
\end{proof}

We now refine the decomposition obtained in Lemma~\ref{lem:closed_form} so that the branching constraints added at each step of branch-and-cut are invariant within a region. For ease of exposition, we assume that branch-and-cut uses a lexicographic variable selection policy. This means that the variable branched on at each node of the search tree is fixed and given by the lexicographic ordering $x_1,\ldots, x_n$. Generalizing the argument to work for other policies, such as the product scoring rule, can be done as in the single-cut case.

\begin{lemma}\label{lemma:multi_branching}
Suppose branch-and-cut uses a lexicographic variable selection policy. Then, there is a set of of at most $K$ hyperplanes, $3n^2K^n(m+2n)^n\tau^{3n}$ degree-$K$ polynomial hypersurfaces, and $nK^n(m+2n)^{2n}\tau^{3n}$ degree-$2K$ polynomial hypersurfaces partitioning $\R^{n+1}$ into connected components such that within each connected component, the branching constraints used at every step of branch-and-cut are invariant.
\end{lemma}

\begin{proof}
Fix a connected component $C$ in the decomposition established in Corollary~\ref{cor:multi_closed_form}. Then, by Corollary~\ref{cor:multi_closed_form}, for each $\sigma$, either $\vec{x}^*_{\LP}(\cut[1],\ldots,\cut[K], \sigma) = \vec{x}^*_{\LP}(\sigma)$ or there exists cuts (without less of generality) labeled by indices $1,\ldots,k\in[K]$ and there exists $F\subseteq M\cup\sigma$ such that $$\vec{x}^*_{\LP}(\cut[1],\ldots,\cut[K],\sigma)[i] = \frac{\det(A^i_{F, \vec{\alpha}_{1},\beta_{1},\ldots, \vec{\alpha}_{k}, \beta_{k},\sigma})}{\det(A_{F,\vec{\alpha}_{1},\ldots,\vec{\alpha}_{k},\sigma})}$$ for all $(\vec{\alpha}, \beta)\in C$ and all $i \in [n]$. Now, if we are at a stage in the branch-and-cut tree where $\sigma$ is the list of branching constraints added so far, and the $i$th variable is being branched on next, the two constraints generated are $$x_i\le\lf\vec{x}^*_{\LP}(\cut[1],\ldots,\cut[K],\sigma)[i]\rf\text{ and } x_i\ge\left\lceil \vec{x}^*_{\LP}(\cut[1],\ldots,\cut[K],\sigma)[i]\right\rceil,$$ respectively. If $C$ is a component where $\vec{x}^*_{\LP}(\cut[1],\ldots,\cut[K], \sigma) = \vec{x}^*_{\LP}(\sigma)$, then there is nothing more to do, since the branching constraints at that point are trivially invariant over $(\vec{\alpha}_1, \beta_1,\ldots,\vec{\alpha}_K, \beta_K)\in C$. Otherwise, in order to further decompose $C$ such that the right-hand-side of these constraints are invariant for every $\sigma$ and every $i = 1,\ldots, n$, we add the two decision boundaries given by $$k\le\frac{\det(A^i_{F, \vec{\alpha}_{1},\beta_{1},\ldots, \vec{\alpha}_{k}, \beta_{k},\sigma})}{\det(A_{F,\vec{\alpha}_{1},\ldots,\vec{\alpha}_{k},\sigma})} \le k+1$$ for every $i$, $\sigma$, and every integer $k = 0,\ldots, \tau-1$, where $\tau = \lceil\max_{\vec{x}\in\cP}\norm{\vec{x}}_{\infty}\rceil$. This ensures that within every connected component of $C$ induced by these boundaries (degree-$K$ polynomial hypersurfaces), $$\lf\vec{x}^*_{\LP}(\vec{\alpha}^T\vec{x}\le\beta,\sigma)[i]\rf = \lf\frac{\det(A^i_{F, \vec{\alpha}_{1},\beta_{1},\ldots, \vec{\alpha}_{k}, \beta_{k},\sigma})}{\det(A_{F,\vec{\alpha}_{1},\ldots,\vec{\alpha}_{k},\sigma})} \rf$$ and $$\left\lceil \vec{x}^*_{\LP}(\vec{\alpha}^T\vec{x}\le\beta,\sigma)[i]\right\rceil = \left\lceil\frac{\det(A^i_{F, \vec{\alpha}_{1},\beta_{1},\ldots, \vec{\alpha}_{k}, \beta_{k},\sigma})}{\det(A_{F,\vec{\alpha}_{1},\ldots,\vec{\alpha}_{k},\sigma})} \right\rceil$$ are invariant, so the branching constraints added by, for example, a lexicographic branching rule, are invariant. For a fixed $\sigma$, there are two hypersurfaces for every subset $S\subseteq[K]$, every $F\subseteq M\cup\sigma$ corresponding to a $|S|$-dimensional face of $\cP(\sigma)$, and every $i = 1,\ldots, n$, for a total of at most $2n^2K^n\binom{m+|\sigma|}{|S|}\le 2n^2K^n(m+2n)^{n}$. Summing over all reduced $\sigma$, we get a total of $2n^2K^n(m+2n)^n\tau^{3n}$ hypersurfaces. Adding these hypersurfaces to the set of hypersurfaces established in Corollary~\ref{cor:multi_closed_form} yields the lemma statement.
\end{proof}

Now, as in the single-cut case, we consider the constraints that ensure that all cuts are valid. Let $\cV\subseteq\R^{K(n+1)}$ denote the set of all vectors of valid $K$ cuts. As before, $\cV$ is a polyhedron, since we may write $$\cV = \bigcap_{k=1}^K\bigcap_{\vec{x}_{\IH}\in\cP_{\IH}}\left\{(\vec{\alpha}_1, \beta_1,\ldots,\vec{\alpha}_K, \beta_k)\in\R^{K(n+1)} : \vec{\alpha}^T_k\vec{x}_{\IH}\le\beta_k\right\}.$$

We now refine our decomposition further to control the integrality of the various LP solutions at each node of branch-and-cut. 

\begin{lemma}\label{lemma:multi_integrality}
Given an IP $(\vec{c}, A, \vec{b})$, there is a set of at most $2K\tau^n$ hyperplanes, $4n^2K^n(m+2n)^n\tau^{4n}$ degree-$K$ polynomial hypersurfaces, and $nK^n(m+2n)^{2n}\tau^{3n}$ degree-$2K$ polynomial hypersurfaces partitioning $\R^{K(n+1)}$ into connected components such that for each component $C$, and each $\sigma\subseteq\cB\cC$, $$\mathbf{1}\left[\vec{x}^*_{\LP}\left(\cut[1],\ldots,\cut[K], \sigma\right)\in\Z^n\right]$$ is invariant for all $(\vec{\alpha}_1, \beta_1,\ldots,\vec{\alpha}_K, \beta_K)\in C$.
\end{lemma}

\begin{proof}
Fix a connected component $C$ in the decomposition that includes the facets defining $\cV$ and the surfaces obtained in Lemma~\ref{lemma:multi_branching}. For all $\sigma\in\cB\cC$, $\vec{x}_{\I}\in\cP_{\I}$, and $i = 1,\ldots, n$, consider the surface \begin{equation}\label{eq:multi_integer_point}\vec{x}^*_{\LP}\left(\cut[1],\ldots,\cut[K],\sigma\right)[i]=\vec{x}_{\I}[i].\end{equation} This surface is a polynomial hypersurface of degree at most $K$, due to Corollary~\ref{cor:multi_closed_form}. Clearly, within any connected component of $C$ induced by these hyperplanes, for every $\sigma$ and $\vec{x}_{\I}\in\cP_{\I}$, $\mathbf{1}[\vec{x}^*_{\LP}(\cut[1],\ldots,\cut[K], \sigma) = \vec{x}_{\I}]$ is invariant. Finally, if $\vec{x}^*_{\LP}(\cut[1],\ldots,\cut[K], \sigma)\in\Z^n$ for some $K$ cuts $\cut[1],\ldots,\cut[K]$ within a given connected component, $\vec{x}^*_{\LP}(\cut[1],\ldots,\cut[K], \sigma) = \vec{x}_{\I}$ for some $\vec{x}_{\I}\in\cP_{\IH}(\sigma)\subseteq\cP_{\I}$, which means that $\vec{x}^*_{\LP}(\cut[1],\ldots,\cut[K], \sigma) = \vec{x}_{\I}\in\Z^n$ \emph{for all} vectors of $K$ cuts $\cut[1],\ldots,\cut[K]$ in that connected component.

We now count the number of hyperplanes given by Equation~\ref{eq:multi_integer_point}. For each $\sigma$, there are $nK^n$ possible subsets of cut indices and at most $(m+2n)^{n-1}$ binding face constraints $F\subseteq M\cup\sigma$ defining the formula of Corollary~\ref{cor:multi_closed_form}. For each subset-face pair, there are $n|\cP_{\I}|\le n\tau^n$ degree-$K$ polynomial hypersurfaces given by Equation~\ref{eq:multi_integer_point}. So the total number of such hypersurfaces over all $\sigma$ is at most $\tau^{3n}n^2K^n(m+2n)^{n-1}\tau^n$. The number of facets defining $\cV$ is at most $K|\cP_{\I}|\le K\tau^n$. Adding these to the counts obtained in Lemma~\ref{lemma:multi_branching} yields the final tallies in the lemma statement.
\end{proof}

At this point, as in the single-cut case, if the bounding aspect of branch-and-cut is suppressed, our decomposition yields connected components over which the branch-and-cut tree built is invariant. We now prove our main structural theorem for B\&C as a function of multiple cutting planes at the root.

\begin{theorem}\label{theorem:multi_tree_invariant}
Given an IP $(\vec{c}, A, \vec{b})$, there is a set of at most $O(12^nn^{2n}K^{2n^2}(m+2n)^{2n^2}\tau^{5n^2})$ polynomial hypersurfaces of degree at most $2K$ partitioning $\R^{K(n+1)}$ into connected components such that the branch-and-cut tree built after adding the $K$ cuts $\cut[1],\ldots,\cut[k]$ at the root is invariant over all $(\vec{\alpha}_1, \beta_1,\ldots,\vec{\alpha}_K,\beta_K)$ within a given component. In particular, $f_{\vec{c}, A, \vec{b}}(\vec{\alpha}_1, \beta_1,\ldots,\vec{\alpha}_K,\beta_K)$ is invariant over each connected component.
\end{theorem}

\begin{proof} Fix a connected component $C$ in the decomposition induced by the set of hyperplanes, degree-$K$ hypersurfaces, and degree-$2K$ hypersurfaces established in Lemma~\ref{lemma:multi_integrality}. Let \begin{equation}\label{eq:multi_node_list}Q_1,\ldots, Q_{i_1}, I_1, Q_{i_1+1},\ldots,Q_{i_2},I_2,Q_{i_2+1},\ldots\end{equation} denote the nodes of the tree branch-and-cut creates, in order of exploration, under the assumption that a node is pruned if and only if either the LP at that node is infeasible or the LP optimal solution is integral (so the ``bounding'' of branch-and-bound is suppressed). Here, a node is identified by the list $\sigma$ of branching constraints added to the input IP. Nodes labeled by $Q$ are either infeasible or have fractional LP optimal solutions. Nodes labeled by $I$ have integral LP optimal solutions and are candidates for the incumbent integral solution at the point they are encountered. (The nodes are functions of $\vec{\alpha}_1,\beta_1,\ldots,\vec{\alpha}_K,\beta_K$, as are the indices $i_1, i_2,\ldots$.) By Lemma~\ref{lemma:multi_integrality}, this ordered list of nodes is invariant for all $(\vec{\alpha}_1, \beta_1,\ldots,\vec{\alpha}_K,\beta_k)\in C$.

Now, given an node index $\ell$, let $I(\ell)$ denote the incumbent node with the highest objective value encountered up until the $\ell$th node searched by B\&C, and let $z(I(\ell))$ denote its objective value. For each node $Q_{\ell}$, let $\sigma_{\ell}$ denote the branching constraints added to arrive at node $Q_{\ell}$. The hyperplane \begin{equation}\label{eq:multi_prune}z^*_{\LP}\left(\cut[1],\ldots,\cut[K],\sigma_{\ell}\right) = z(I(\ell))\end{equation} (which is a hyperplane due to Corollary~\ref{cor:multi_closed_form}) partitions $C$ into two subregions. In one subregion, $z^*_{\LP}(\cut[1],\ldots,\cut[k],\sigma_{\ell}) \le z(I(\ell))$, that is, the objective value of the LP optimal solution is no greater than the objective value of the current incumbent integer solution, and so the subtree rooted at $Q_{\ell}$ is pruned. In the other subregion, $z^*_{\LP}(\cut[1],\ldots,\cut[k],\sigma_{\ell}) > z(I(\ell))$, and $Q_{\ell}$ is branched on further. Therefore, within each connected component of $C$ induced by all hyperplanes given by Equation~\ref{eq:multi_prune} for all $\ell$, the set of node within the list~(\ref{eq:multi_node_list}) that are pruned is invariant. Combined with the surfaces established in Lemma~\ref{lemma:multi_integrality}, these hyperplanes partition $\R^{K(n+1)}$ into connected components such that as $(\vec{\alpha}_1, \beta_1\ldots,\vec{\alpha}_K,\beta_K)$ varies within a given component, the tree built by branch-and-cut is invariant.

Finally, we count the total number of surfaces inducing this partition. Unlike the counting stages of the previous lemmas, we will first have to count the number of connected components induced by the surfaces established in Lemma~\ref{lemma:multi_integrality}. This is because the ordered list of nodes explored by branch-and-cut~(\ref{eq:multi_node_list}) can be different across each component, and the hyperplanes given by Equation~\ref{eq:multi_prune} depend on this list. From Lemma~\ref{lemma:multi_integrality} we have $6n^2K^n(m+2n)^{2n}\tau^{4n}$ polynomial hypersurfaces of degree $\le 2K$. The set of all $(\vec{\alpha}_1, \beta_1,\ldots\vec{\alpha}_K,\beta_k)\in\R^{K(n+1)}$ such that $(\vec{\alpha}_1, \beta_1,\ldots,\vec{\alpha}_K,\beta_K)$ lies on the boundary of any of these surfaces is precisely the zero set of the product of all polynomials defining these surfaces. Denote this product polynomial by $p$. The degree of the product polynomial is the sum of the degrees of $6n^2K^n(m+2n)^{2n}\tau^{4n}$ polynomials of degree $\le 2K$, which is at most $2K\cdot 6Kn^2K^n(m+2n)^{2n}\tau^{4n} = 12n^2K^{n+2}(m+2n)^{2n}\tau^{4n}$. By Warren's theorem, the number of connected components of $\R^{n+1}\setminus\{(\vec{\alpha}, \beta) : p(\vec{\alpha}, \beta) = 0\}$ is $O((12n^2K^{n+2}(m+2n)^{2n}\tau^{4n})^{n-1})$, and by the Milnor-Thom theorem, the number of connected components of $\{(\vec{\alpha}, \beta) : p(\vec{\alpha}, \beta) = 0\}$ is $O((12n^2K^{n+2}(m+2n)^{2n}\tau^{4n})^{n-1})$ as well. So, the number of connected components induced by the surfaces in Lemma~\ref{lemma:multi_integrality} is $O(12^nn^{2n}K^{2n^2}(m+2n)^{2n^2}\tau^{4n^2}).$ For every connected component $C$ in Lemma~\ref{lemma:multi_integrality}, the closed form of $z^{*}_{\LP}(\cut, \sigma_{\ell})$ is already determined due to Corollary~\ref{cor:multi_closed_form}, and so the number of hyperplanes given by Equation~\ref{eq:multi_prune} is at most the number of possible $\sigma\subseteq\cB\cC$, which is at most $\tau^{3n}$. So across all connected components $C$, the total number of hyperplanes given by Equation~\ref{eq:multi_prune} is $O(12^nn^{2n}K^{2n^2}(m+2n)^{2n^2}\tau^{5n^2}).$ Finally, adding this to the surface-counts established in Lemma~\ref{lemma:multi_integrality} yields the theorem statement. \end{proof}

\section{Omitted results from Section~\ref{sec:sample}}\label{app:sample}

\begin{proof}[Proof of Theorem~\ref{thm:impossible}]
For a set $\cX$, $\cX^{<\N}$ denotes the set of finite sequences of elements from $\cX$. There is a bijection between the set of IPs $(\vec{c}, A, \vec{b})\in \cI := \R^n\times\Z^{m\times n}\times\Z^m$ and $\R$, so IPs can be uniquely represented as real numbers (and vice versa). Now, consider the set of all finite sequences of pairs of IPs and $\pm 1$ labels of the form $((\vec{c_1}, A_1, \vec{b}_1), \varepsilon_1),\ldots, ((\vec{c_N}, A_N, \vec{b}_N), \varepsilon_N)$, $\varepsilon_1,\ldots,\varepsilon_N\in\{-1, 1\}$, that is, the set $(\cI\times\{-1,1\})^{<\N}$. There is a bijection between this set and $(\R\times\{-1,1\})^{<\N}$, and in turn there is a bijection between $(\R\times\{-1, 1\})^{<\N}$ and $\R$. Hence, there exists a bijection between $\cU$ and $(\cI\times\{-1,1\})^{<\N}$. Fix such a bijection $\varphi:\cU\to (\cI\times\{-1,1\})^{<\N}$, and let $\varphi^{-1}:(\cI\times\{-1,1\})^{<\N}\to\cU$ denote the inverse of $\varphi$, which is well defined and also a bijection.

Let $n$ be odd. For $c\in\R$, let $\IP_c\in\cI$ denote the IP \begin{equation}\begin{array}{ll} \text{maximize} & c\\
\text{subject to} & 2x_1 + \cdots + 2x_n = n\\
& \vec{x} \in \{0,1\}^n.
\end{array}\label{eq:Jeroslow}\end{equation} Since $n$ is odd, $\IP_c$ is infeasible, independent of $c$. Jeroslow~\cite{Jeroslow74:Trivial} showed that without the use of cutting planes or heuristics, branch-and-bound builds a tree of size $2^{(n-1)/2}$ before determining infeasibility and terminating. The objective $c$ is irrelevant, but is important in generating distinct IPs with this property. Consider the cut $x_1+\cdots+x_n\le\lf n/2\rf$, which is a valid cut for $\IP_c$ (this is in fact a Chv\'{a}tal-Gomory cut~\cite{Balcan21:Sample}). In particular, since $n$ is odd, $x_1+\cdots+x_n\le\lf n/2\rf\implies x_1+\cdots+x_n\le (n-1)/2 < n/2$, so the equality constraint of $\IP_c$ is violated by this cut. Thus, the feasible region of the LP relaxation after adding this cut is empty, and branch-and-bound will terminate immediately at the root (building a tree of size $1$). Denote this cut by $(\vec{\alpha}^{(-1)}, \beta^{(-1)}) = (\vec{1}, \lf n/2\rf)$. On the other hand, let $(\vec{\alpha}^{(1)}, \beta^{(1)}) = (\vec{0}, 0)$ be the trivial cut $0\le 0$. Adding this cut to the IP constraints does not change the feasible region, so branch-and-bound will build a tree of size $2^{(n-1)/2}$.

We now define $\vec{\alpha}_{\vec{c}, A, \vec{b}}$ and $\beta_{\vec{c}, A, \vec{b}}$. Let $$(\vec{\alpha}_{\vec{c}, A, \vec{b}}(\vec{u}), \beta_{\vec{c}, A, \vec{b}}(\vec{u})) = \begin{cases}(\vec{\alpha}^{(1)}, \beta^{(1)})  & \text{if } ((\vec{c}, A, \vec{b}), 1)\in\varphi(\vec{u}) \text{ and } ((\vec{c}, A, \vec{b}), -1)\notin\varphi(\vec{u}) \\ (\vec{\alpha}^{(-1)}, \beta^{(-1)})  & \text{if } ((\vec{c}, A, \vec{b}), -1)\in\varphi(\vec{u}) \text{ and } ((\vec{c}, A, \vec{b}), 1)\notin\varphi(\vec{u}) \\ (\vec{0}, 0) & \text{otherwise}\end{cases}.$$
The choice to use $(\vec{0}, 0)$ in the case that either $((\vec{c}, A, \vec{b}), \varepsilon)\notin\varphi(\vec{u})$ for each $\varepsilon\in\{-1,1\}$, or $((\vec{c}, A, \vec{b}), -1)\in\varphi(\vec{u})$ and $((\vec{c}, A, \vec{b}), 1)\in\varphi(\vec{u})$ is arbitrary and unimportant. Now, for any integer $N>0$, constructing a set of $N$ IPs and $N$ thresholds that is shattered is almost immediate. Let $c_1,\ldots, c_N\in\R$ be distinct reals, and let $1 < r_1,\ldots, r_N < 2^{(n-1)/2}$. Then, the set $\{(\IP_{c_1}, r_1),\ldots, (\IP_{c_N}, r_N)\}$ can be shattered. Indeed, given a sign pattern $(\varepsilon_1,\ldots, \varepsilon_N)\in\{-1,1\}^N$, let $$\vec{u} = \varphi^{-1}\left((\IP_{c_1}, \varepsilon_1),\ldots, (\IP_{c_N}, \varepsilon_N)\right).$$ Then, if $\varepsilon_i = 1$, $(\vec{\alpha}_{\IP_{c_i}}(\vec{u}), \beta_{\IP_{c_i}}(\vec{u})) = (\vec{\alpha}^{(1)}, \beta^{(1)})$, so $g_{\vec{u}}(\IP_{c_i}) = 2^{(n-1)/2}$ and $\sign(g_{\vec{u}}(\IP_{c_i}) - r_i) =1$. If $\varepsilon_i = -1$,  $(\vec{\alpha}_{\IP_{c_i}}(\vec{u}), \beta_{\IP_{c_i}}(\vec{u})) = (\vec{\alpha}^{(-1)}, \beta^{(-1)})$, so $g_{\vec{u}}(\IP_{c_i}) = 1$ and $\sign(g_{\vec{u}}(\IP_{c_i}) - r_i) = -1$. So for any $N$ there is a set of IPs and thresholds that can be shattered, which yields the theorem statement.\end{proof}

\begin{proof}[Proof of Lemma~\ref{lemma:gmi_hyperplanes}]
We have $f_i = \vec{u}^T\vec{a}_i - \lfloor\vec{u}^T\vec{a}_i\rfloor$, $f_0 = \vec{u}^T\vec{b}-\lfloor\vec{u}^T\vec{b}\rfloor$, and since $\vec{u}\in [-U, U]^m$, $\lfloor\vec{u}^T\vec{a}_i\rfloor\in [-U\norm{\vec{a}_i}_1,U\norm{\vec{a}_i}_1]$ and $\lfloor\vec{u}^T\vec{b}\rfloor\in [-U\norm{\vec{b}}_1, U\norm{\vec{b}}_1].$ Now, for all $i$, $k_i\in[-U\norm{\vec{a}_i}_1,U\norm{\vec{a}_i}_1]\cap\Z$ and $k_0\in[-U\norm{\vec{b}}_1,U\norm{\vec{b}}_1]\cap\Z$, put down the hyperplanes defining the two halfspaces \begin{equation}\label{eq:gmi_floor1}\lf\vec{u}^T\vec{a}_i\rf = k_i \iff k_i\le \vec{u}^T\vec{a}_i < k_i + 1\end{equation} and the hyperplanes defining the two halfspaces \begin{equation}\label{eq:gmi_floor2}\lf\vec{u}^T\vec{b}\rf = k_0 \iff k_0\le \vec{u}^T\vec{b} < k_0 + 1.\end{equation} In addition, consider the hyperplane \begin{equation}\label{eq:gmi_form}\vec{u}^T\vec{a}_i - k_i = \vec{u}^T\vec{b}-k_0\end{equation} for each $i$. Within any connected component of $\R^m$ determined by these hyperplanes, $\lfloor\vec{u}^T\vec{a}_i\rfloor$ and $\lfloor\vec{u}^T\vec{b}\rfloor$ are constant. Furthermore, $\mathbf{1}[f_i\le f_0]$ is invariant within each connected component, since if $\lfloor\vec{u}^T\vec{a}_i\rfloor = k_i$ and $\lfloor\vec{u}^T\vec{b}\rfloor = k_0$, $f_i\le f_0\iff \vec{u}^T\vec{a}_i - k_i \le \vec{u}^T\vec{b}-k_0,$ which is the hyperplane given by Equation~\ref{eq:gmi_form}. The total number of hyperplanes of type~\ref{eq:gmi_floor1} is $O(nU\norm{A}_1)$, the total number of hyperplanes of type~\ref{eq:gmi_floor2} is $O(U\norm{\vec{b}}_1)$, and the total number of hyperplanes of type~\ref{eq:gmi_form} is $nU^2\norm{A}_1\norm{\vec{b}}_1$. Summing yields the lemma statement.
\end{proof}

\begin{proof}[Proof of Lemma~\ref{lemma:gmi_tree_invariant}]
Let $C\subseteq\R^{n+1}$ be a connected component in the partition established in Theorem~\ref{thm:tree_invariant}, so $C$ can be written as the intersection of at most $14^n(m+2n)^{3n^2}\tau^{5n^2}$ polynomial constraints of degree at most $5$. Let $D\subseteq[-U,U]^m$ be a connected component in the partition established in Lemma~\ref{lemma:gmi_hyperplanes}. By Lemma~\ref{lemma:gmi_polynomial}, there are at most $14^n(m+2n)^{3n^2}\tau^{5n^2}$ polynomials of degree at most $10$ partitioning $D$ into connected components such that within each component, $\mathbf{1}[(\vec{\alpha}(\vec{u}), \beta(\vec{u}))\in C]$ is invariant. If we consider the overlay of these polynomial surfaces over all components $C$, we will get a partition of $[-U,U]^m$ such that \emph{for every} $C$, $\mathbf{1}[(\vec{\alpha}(\vec{u}), \beta(\vec{u}))\in C]$ is invariant over each connected component of $[-U,U]^m$. Once we have this we are done, since all $\vec{u}$ in the same connected component of $[-U, U]^m$ will be sent to the same connected component of $\R^{n+1}$ by $(\vec{\alpha}(\vec{u}), \beta(\vec{u}))$, and thus by Theorem~\ref{thm:tree_invariant} the behavior of branch-and-cut will be invariant.

We now tally up the total number of surfaces. The number of connected components $C$ was given by Warren's theorem and the Milnor-Thom theorem to be $O(14^{n(n+1)}(m+2n)^{3n^2(n+1)}\tau^{5n^2(n+1)})$, so the total number of degree-$10$ hypersurfaces is $14^n(m+2n)^{3n^2}\tau^{5n^2}$ times this quantity, which yields the lemma statement.\end{proof}

\subsection{Multiple GMI cuts at the root}

In this section we extend our results to allow for multiple GMI cuts at the root of the B\&C tree. These cuts can be added simultaneously, sequentially, or in rounds. If GMI cuts $\vec{u}_1$, $\vec{u}_2$ are added simultaneously, both of them have the same dimension and are defined in the usual way. If GMI cuts $\vec{u}_1$, $\vec{u}_2$ are added sequentially, $\vec{u}_2$ has one more entry than $\vec{u}_1$. This is because when cuts are added sequentially, the LP relaxation is re-solved after the addition of the first cut, and the second cut has a multiplier for all original constraints as well as for the first cut (this ensures that the second cut can be chosen in a more informed manner). If $K$ cuts are made at the root, they can be added in sequential rounds of simultaneous cuts. In the following discussion, we focus on the case where all $K$ cuts are added sequentially---the other cases can be viewed as instantiations of this. We refer the reader to the discussion in~\citet{Balcan21:Sample} for more details.

To prove an analogous result for multiple GMI cuts (in sequence, that is, each successive GMI cut has one more parameter than the previous), we combine the reasoning used in the single-GMI-cut case with some technical observations in~\citet{Balcan21:Sample}.

\begin{lemma}\label{lemma:multi_gmi_tree_invariant}
Consider the family of $K$ sequential GMI cuts parameterized by $\vec{u}_1\in[-U,U]^m,\vec{u}_2\in[-U,U]^{m+1},\ldots,\vec{u}_K\in[-U,U]^{m+K-1}$. For any IP $(\vec{c},A,\vec{b})$, there are at most $$O\left(nK(1+U)^{2K}\norm{A}_1\norm{\vec{b}}_1\right)$$ degree-$K$ polynomial hypersurfaces and $$2^{O(n^2)}K^{O(n^3)}(m+2n)^{O(n^3)}\tau^{O(n^3)}$$ degree-$4K^2$ polynomial hypersurfaces partitioning $[-U,U]^m\times\cdots\times[-U,U]^{m+K-1}$ connected components such that the B\&C tree built after sequentially adding the GMI cuts defined by $\vec{u}_1,\ldots,\vec{u}_K$ is invariant over all $(\vec{u}_1,\ldots,\vec{u}_K)$ within a single component.
\end{lemma}

\begin{proof} We start with the setup used by~\citet{Balcan21:Sample} to prove similar results for sequential Chv\'{a}tal-Gomory cuts. Let $\vec{a}_1,\ldots,\vec{a}_n\in\R^m$ be the columns of $A$. We define the following augmented columns $\widetilde{\vec{a}}_i^1\in\R^m,\ldots,\widetilde{\vec{a}}_i^K\in\R^{m+K-1}$ for each $i\in [n]$, and the augmented constraint vectors $\widetilde{\vec{b}}^1\in\R^m,\ldots,\widetilde{\vec{b}}^K\in\R^{m+K-1}$ via the following recurrences: \begin{align*}
    \vec{\widetilde{a}}^1_i &= \vec{a}_i \\
    \vec{\widetilde{a}}^k_i &= \begin{bmatrix}\vec{\widetilde{a}}^{k-1}_i \\ \vec{u}^T_{k-1}\vec{\widetilde{a}}^{k-1}_i \end{bmatrix}
\end{align*} and \begin{align*}
    \vec{\widetilde{b}}^1 &= \vec{b} \\
    \vec{\widetilde{b}}^k &= \begin{bmatrix}\vec{\widetilde{b}}^{k-1} \\ \vec{u}_{k-1}^T\vec{\widetilde{b}}^{k-1}\end{bmatrix}
\end{align*} for $k = 2,\ldots, K$. In other words, $\widetilde{\vec{a}}^k_i$ is the $i$th column of the constraint matrix of the IP and $\widetilde{\vec{b}}^k$ is the constraint vector after applying cuts $\vec{u}_1,\ldots,\vec{u}_{k-1}$. An identical induction argument to that of~\citet{Balcan21:Sample} shows that for each $k\in [K]$, $$\big\lfloor\vec{u}^T_k\widetilde{\vec{a}}_i^k\big\rfloor\in\left[-\left(1+U\right)^{k}\norm{\vec{a}_i}_1,\left(1+U\right)^{k}\norm{\vec{a}_i}_1\right]$$ and $$\big\lfloor\vec{u}^T_k\widetilde{\vec{b}}^k\big\rfloor\in\left[-\left(1+U\right)^{k}\norm{\vec{b}}_1,\left(1+U\right)^{k}\norm{\vec{b}}_1\right].$$ Now, as in the single-GMI-cut setting, consider the surfaces \begin{equation}\label{eq:multi_gmi_1}\big\lfloor\vec{u}_k^T\widetilde{\vec{a}}^k_i\big\rfloor = \ell_i\iff \ell_i\le \vec{u}_k^T\widetilde{\vec{a}}^k_i < \ell_i + 1\end{equation} and \begin{equation}\label{eq:multi_gmi_2}\big\lfloor\vec{u}_k^T\widetilde{\vec{b}}^k\big\rfloor = \ell_0\iff \ell_i\le \vec{u}_k^T\widetilde{\vec{b}}^k < \ell_0 + 1\end{equation} for every $i, k$, and every integer $\ell_i\in [-(1+U)^{k}\norm{\vec{a}_i}_1,(1+U)^{k}\norm{\vec{a}_i}_1]\cap\Z$ and every integer $\ell_0\in [-(1+U)^{k}\norm{\vec{b}}_1,(1+U)^{k}\norm{\vec{b}}_1]\cap\Z$. In addition, consider the surfaces \begin{equation}\label{eq:multi_gmi_3}\vec{u}_k^T\widetilde{\vec{a}}_i^k - \ell_i =  \vec{u}_k^T\widetilde{\vec{b}}^k - \ell_0\end{equation} for each $i, k, \ell_i,\ell_0$. As observed by~\citet{Balcan21:Sample}, $\vec{u}_k^T\widetilde{\vec{a}}^k_i$ is a polynomial in $\vec{u}_1[1],\ldots,\vec{u}_1[m], \vec{u}_2[1],\ldots,\vec{u}_2[m+1],\ldots,\vec{u}_k[1],\ldots,\vec{u}_k[m+k-1]$ of degree at most $k$ (as is $\vec{u}_k^T\widetilde{\vec{b}}^k$), so surfaces~\ref{eq:multi_gmi_1},~\ref{eq:multi_gmi_2}, and~\ref{eq:multi_gmi_3} are all degree-$K$ polynomial hypersurfaces for all $i, k$. Within any connected component of $[-U,U]^m\times\cdots\times[-U,U]^{m+K-1}$ induced by these hypersurfaces, $\lfloor\vec{u}_k^T\widetilde{\vec{a}}^k_i\rfloor$ and $\lfloor\vec{u}_k^T\widetilde{\vec{b}}^k\rfloor$ are constant. Furthermore $\mathbf{1}[f^k_i\le f^k_0]$ is invariant for every $i, k$, where $f^k_i =\vec{u}_k^T\widetilde{\vec{a}}_i^k- \lfloor\vec{u}_k^T\widetilde{\vec{a}}_i^k\rfloor$ and $f^k_0 =\vec{u}_k^T\widetilde{\vec{b}}^k- \lfloor\vec{u}_k^T\widetilde{\vec{b}}^k\rfloor$.

Now, fix a connected component $D\subseteq [-U,U]^m\times\cdots\times[-U,U]^{m+K-1}$ induced by the above hypersurfaces, and let $C\subseteq\R^{K(n+1)}$ be the intersection of $q$ polynomial inequalities of degree at most $d$. Consider a single degree-$d$ polynomial inequality in $K(n+1)$ variables $y_1,\ldots, y_{K(n+1)}$, which can be written as $$\sum_{\substack{T\sqsubseteq[K(n+1)] \\ |T|\le d}}\lambda_T\prod_{j\in T}y_j = \sum_{\substack{T_1,\ldots,T_K\sqsubseteq[n+1] \\ |T_1|+\cdots+|T_K|\le d}}\lambda_{T_1,\ldots,T_K}\prod_{j_1\in T_1}y_{j_1}\cdots\prod_{j_K\in T_K}y_{j_K}\le\gamma.$$ Now, the sets $S_1,\ldots, S_K$ defined by $S_k = \{i : f^k_i\le f^k_0\}$ are fixed within $D$, so we can write this as $$\sum_{\substack{T_1,\ldots,T_K\sqsubseteq[n+1] \\ |T_1|+\cdots+|T_K|\le d}}\lambda_{T_1,\ldots,T_K}\prod_{k=1}^K\Bigg[\prod_{\substack{j\in T_k\cap S_k \\ j\neq n+1}}f^k_j(1-f^k_0)\prod_{\substack{j\in T_k\setminus S_k \\ j\neq n+1}}f^k_0(1-f^k_j)\prod_{\substack{j\in T_k \\ j = n+1}} f^k_0(1-f^k_0)\Bigg]\le\gamma.$$ We have that $f_j^k$ and $f_0^k$ are degree-$k$ polynomials in $\vec{u}_1,\ldots,\vec{u}_k$. Since the sum is over all multisets $T_1,\ldots, T_K$ such that $|T_1|+\cdots+|T_K|\le d$, there are at most $d$ terms across the products, each of the form $f_j^k(1-f_0)^k$, $f_0^k(1-f_j^k)$, or $f_0^k(1-f_0)^k$. Therefore, the left-hand-side is a polynomial of degree at most $2dK$, and if $C\subseteq\R^{K(n+1)}$ is the intersection of $q$ polynomial inequalities each of degree at most $d$, the set $$\left\{\left(\vec{u}_1,\ldots,\vec{u}_K\right)\in D:\left(\vec{\alpha}\left(\vec{u}_1,\ldots,\vec{u}_K\right),\beta\left(\vec{u}_1,\ldots,\vec{u}_K\right)\right)\in C\right\}\subseteq [-U,U]^m\times\cdots\times[-U,U]^{m+K-1}$$ can be expressed as the intersection of $q$ degree-$2dK$ polynomial inequalities.

To finish, we run this process for every connected component $C\subseteq\R^{K(n+1)}$ in the partition established by Theorem~\ref{theorem:multi_tree_invariant}. This partition consists of $O(12^nn^{2n}K^{2n^2}(m+2n)^{2n^2}\tau^{5n^2})$ degree-$2K$ polynomials over $\R^{K(n+1)}$. By Warren's theorem and the Milnor-Thom theorem, these polynomials partition $\R^{K(n+1)}$ into $O(12^{n(n+1)}n^{2n(n+1)}K^{2n^2(n+1)}(m+2n)^{2n^2(n+1)}\tau^{5n^2(n+1)})$ connected components. Running the above argument for each of these connected components of $\R^{K(n+1)}$ yields a total of $O\left(12^{n(n+1)}n^{2n(n+1)}K^{2n^2(n+1)}(m+2n)^{2n^2(n+1)}\tau^{5n^2(n+1)}\right)\cdot O\left(12^nn^{2n}K^{2n^2}(m+2n)^{2n^2}\tau^{5n^2}\right) = 2^{O(n^2)}K^{O(n^3)}(m+2n)^{O(n^3)}\tau^{O(n^3)}$ polynomials of degree $4K^2$. Finally, we count the surfaces of the form~\eqref{eq:multi_gmi_1}, ~\eqref{eq:multi_gmi_2}, and~\eqref{eq:multi_gmi_3}. The total number of degree-$K$ polynomials of type~\ref{eq:multi_gmi_1} is at most $O(nK(1+U)^K\norm{A}_1)$, the total number of degree-$k$ polynomials of type~\ref{eq:multi_gmi_2} is $O(K(1+U)^K\norm{\vec{b}}_1)$, and the total number of degree-$K$ polynomials of type~\ref{eq:multi_gmi_3} is $O(nK(1+U)^{2K}\norm{A}_1\norm{\vec{b}})$. Summing these counts yields the desired number of surfaces in the lemma statement.

In any connected component of $[-U,U]^m$ determined by these surfaces, $\mathbf{1}[(\vec{\alpha}(\vec{u}),\beta(\vec{u}))\in C]$ is invariant for every connected component $C\subseteq\R^{K(n+1)}$ in the partition of $\R^{K(n+1)}$ established in Theorem~\ref{theorem:multi_tree_invariant}. This means that the tree built by branch-and-cut is invariant, which concludes the proof. \end{proof}

Finally, applying the main result of~\citet{Balcan21:How} to Lemma~\ref{lemma:multi_gmi_tree_invariant}, we get the following pseudo-dimension bound for the class of $K$ sequential GMI cuts at the root of the B\&C tree.

\begin{theorem}
For $\vec{u}_1\in[-U,U]^m,\vec{u}_2\in[-U,U]^{m+1},\ldots,\vec{u}_K\in[-U,U]^{m+K-1}$, let $g_{\vec{u}_1,\ldots,\vec{u}_K}(\vec{c},A,\vec{b})$ denote the number of nodes in the tree B\&C builds given the input $(\vec{c},A,\vec{b})$ after sequentially applying the GMI cuts defined by $\vec{u}_1,\ldots,\vec{u}_K$ at the root. The pseudo-dimension of the set of functions $\{g_{\vec{u_1},\ldots,\vec{u}_K} : (\vec{u}_1,\ldots,\vec{u}_K)\in[-U,U]^m\times\cdots\times[-U,U]^{m+K-1}\}$ on the domain of IPs with $\norm{A}_1\le a$ and $\norm{\vec{b}}_1\le b$ is $$O\left(mK^3\log U + mn^3K^2\log(mnK\tau) + mK^2\log(ab)\right).$$
\end{theorem}

\end{document}